\documentclass[12pt,reqno]{amsart}
\usepackage{mathrsfs}
\usepackage{mathabx}
\usepackage{enumerate, pifont}
\usepackage{bbm}
\usepackage{skull}
\usepackage{mathrsfs, mathtools}
\usepackage{stmaryrd}
\usepackage{amsmath,amssymb,amsfonts, mathabx}
\usepackage[all,cmtip]{xy}
\usepackage{a4wide}
\usepackage[mathscr]{eucal}
\usepackage[toc,page]{appendix}
\usepackage{verbatim}
\usepackage{dsfont,color}

\usepackage{extarrow}
\usepackage{enumerate}
\usepackage{fontenc}
\usepackage[T1]{fontenc}
\usepackage{graphicx}
\usepackage[colorlinks=true, linkcolor=blue]{hyperref}
\usepackage{latexsym}
\usepackage{mathrsfs}
\usepackage{mathtext}
\usepackage{tikz}
\usepackage{textcomp}
\usepackage{upgreek}
\usepackage{xy}

\usetikzlibrary{arrows}
\usetikzlibrary{matrix}
\usetikzlibrary{shapes}
\usetikzlibrary{snakes}
\usetikzlibrary{matrix}

\DeclareMathOperator{\Autsh}{\underline{\textup{Aut}}}
\DeclareMathOperator{\Bi}{\textup{B}}

\DeclareMathOperator{\Coh}{\textup{Coh}}

\DeclareMathOperator{\GL}{\textup{GL}}

\DeclareMathOperator{\Hl}{\textup{H}}
\DeclareMathOperator{\Homsh}{\underline{\textup{Hom}}}

\DeclareMathOperator{\Imm}{\textup{Im}}

\DeclareMathOperator{\Isosh}{\underline{\textup{Iso}}}

\DeclareMathOperator{\QCoh}{\textup{QCoh}}

\DeclareMathOperator{\car}{\textup{char}}

\DeclareMathOperator{\et}{\textup{et}}

\DeclareMathOperator{\pr}{\textup{pr}}

\DeclareMathOperator{\rk}{\textup{rk}}

\usepackage[colorinlistoftodos]{todonotes}

\theoremstyle{plain}
\newtheorem{thm}{Theorem}[section]
\newtheorem{lem}[thm]{Lemma}

\newtheorem{prop}[thm]{Proposition}

\theoremstyle{definition}

\newtheorem{defn}[thm]{Definition} 
\newtheorem{ex}[thm]{Example}

\newtheorem{rmk}[thm]{Remark}

\numberwithin{thm}{section}
\newcounter{x}

\renewcommand{\et}{\textup{\'et}}

\newcommand{\red}{{\rm red}}

\newcommand{\EF}{{\rm EFin}}

\newcommand{\Ess}{{\rm EFin}}

\newcommand{\Hom}{{\rm Hom}}

\newcommand{\Spec}{{\rm Spec \,}}

\newcommand{\Aff}{{\rm Aff}}

\newcommand{\Ker}{{\rm Ker}}
\newcommand{\Aut}{{\rm Aut}}

\newcommand{\sC}{{\mathcal C}}

\newcommand{\sE}{{\mathcal E}}

\newcommand{\sG}{{\mathcal G}}

\newcommand{\sO}{{\mathcal O}}
\newcommand{\sP}{{\mathcal P}}

\newcommand{\sT}{{\mathcal T}}
\newcommand{\sU}{{\mathcal U}}

\newcommand{\sX}{{\mathcal X}}
\newcommand{\sY}{{\mathcal Y}}
\newcommand{\sZ}{{\mathcal Z}}

\newcommand{\A}{{\mathbb A}}

\newcommand{\C}{{\mathbb C}}

\newcommand{\E}{{\mathbb E}}
\newcommand{\F}{{\mathbb F}}
\newcommand{\G}{{\mathbb G}}

\newcommand{\N}{{\mathbb N}}

\newcommand{\Q}{{\mathbb Q}}

\newcommand{\Z}{{\mathbb Z}}

\newcommand{\NN}{\textup{N}}
\renewcommand{\SS}{\textup{S}}

\newcommand{\Mod}{\text{\sf Mod}}

\newcommand{\Vect}{\text{\sf Vect}}
\newcommand{\Rep}{\text{\sf Rep\,}}
\newcommand{\id}{{\rm id\hspace{.1ex}}}

\newcommand{\Ns}{\textup{Ns}}

\theoremstyle{plain}
\newtheorem{thmI}{Theorem}
\newtheorem{thmII}{Theorem}
\newtheorem{thmIII}{Theorem}
\newtheorem{thmIV}{Theorem}
\newtheorem{corI}{Corollary}

\usepackage{anyfontsize}

\begin{document}

\title[Essentially finite covers and towers of torsors]{Nori fundamental gerbe of 
essentially finite covers \\ and Galois closure of towers of torsors}

\author[M. Antei]{Marco Antei}

\address{
Universidad de Costa Rica, Ciudad universitaria Rodrigo Facio Brenes, Costa Rica}
\email{marco.antei@ucr.ac.cr}

\author[I. Biswas]{Indranil Biswas}

\address{
    School of Mathematics, Tata Institute of Fundamental
Research, Homi Bhabha Road, Bombay 400005, India }
\email{indranil@math.tifr.res.in}

\author[M. Emsalem]{Michel Emsalem}

\address{
    Laboratoire Paul Painlev\'e, U.F.R. de Math\'ematiques, Universit\'e des 
Sciences et des Technologies de Lille 1, 59 655 Villeneuve d'Ascq, France }
\email{emsalem@math.univ-lille1.fr}

\author[F. Tonini]{Fabio Tonini}

\address{
    Universit\'a degli Studi di Firenze\\
    Dipartimento di Matematica e Informatica 'Ulisse Dini'\\
    Viale Morgagni, 67/a\\ Firenze, 50134 Italy }
\email{fabio.tonini@unifi.it}

\author[L. Zhang]{Lei Zhang}

 \address{
    The Chinese University of Hong Kong\\
    Department of Mathematics\\    
    Shatin, New Territories\\ Hong Kong }
\email{lzhang@math.cuhk.edu.hk}

\thanks{This work is supported by the departmental fund of the Chinese
    University of Hong Kong and the Labex CEMPI 
(ANR-11-LABX-01). The second author is supported by  the J. C. Bose Fellowship}
\date{\today}

\subjclass[2010]{14F35, 14D23.}

\keywords{Nori fundamental gerbe, essentially finite bundle, essentially finite cover,
algebraic stack.}

\begin{abstract}
We prove the existence of a Galois closure for towers of torsors under finite group schemes over a proper, geometrically connected and geometrically reduced algebraic stack $X$ over a field $k$. This is done by describing the Nori fundamental gerbe of an essentially finite cover of $X$. A similar result is also obtained for the $\SS$-fundamental gerbe.
\end{abstract}

\global\long\def\A{\mathbb{A}}

\global\long\def\Ab{(\textup{Ab})}

\global\long\def\C{\mathbb{C}}

\global\long\def\Cat{(\textup{cat})}

\global\long\def\Di#1{\textup{D}(#1)}

\global\long\def\E{\mathcal{E}}

\global\long\def\F{\mathbb{F}}

\global\long\def\GCov{G\textup{-Cov}}

\global\long\def\Gcat{(\textup{Galois cat})}

\global\long\def\Gfsets#1{#1\textup{-fsets}}

\global\long\def\Gm{\mathbb{G}_{m}}

\global\long\def\GrCov#1{\textup{D}(#1)\textup{-Cov}}

\global\long\def\Grp{(\textup{Grps})}

\global\long\def\Gsets#1{(#1\textup{-sets})}

\global\long\def\HCov{H\textup{-Cov}}

\global\long\def\MCov{\textup{D}(M)\textup{-Cov}}

\global\long\def\MHilb{M\textup{-Hilb}}

\global\long\def\N{\mathbb{N}}

\global\long\def\PGor{\textup{PGor}}

\global\long\def\PGrp{(\textup{Profinite Grp})}

\global\long\def\PP{\mathbb{P}}

\global\long\def\Pj{\mathbb{P}}

\global\long\def\Q{\mathbb{Q}}

\global\long\def\RCov#1{#1\textup{-Cov}}

\global\long\def\RR{\mathbb{R}}

\global\long\def\Sch{\textup{Sch}}

\global\long\def\WW{\textup{W}}

\global\long\def\Z{\mathbb{Z}}

\global\long\def\acts{\curvearrowright}

\global\long\def\alA{\mathscr{A}}

\global\long\def\alB{\mathscr{B}}

\global\long\def\arr{\longrightarrow}

\global\long\def\arrdi#1{\xlongrightarrow{#1}}

\global\long\def\catC{\mathscr{C}}

\global\long\def\catD{\mathscr{D}}

\global\long\def\catF{\mathscr{F}}

\global\long\def\catG{\mathscr{G}}

\global\long\def\comma{,\ }

\global\long\def\covU{\mathcal{U}}

\global\long\def\covV{\mathcal{V}}

\global\long\def\covW{\mathcal{W}}

\global\long\def\duale#1{{#1}^{\vee}}

\global\long\def\fasc#1{\widetilde{#1}}

\global\long\def\fsets{(\textup{f-sets})}

\global\long\def\iL{r\mathscr{L}}

\global\long\def\id{\textup{id}}

\global\long\def\la{\langle}

\global\long\def\odi#1{\mathcal{O}_{#1}}

\global\long\def\ra{\rangle}

\global\long\def\set{(\textup{Sets})}

\global\long\def\sets{(\textup{Sets})}

\global\long\def\shA{\mathcal{A}}

\global\long\def\shB{\mathcal{B}}

\global\long\def\shC{\mathcal{C}}

\global\long\def\shD{\mathcal{D}}

\global\long\def\shE{\mathcal{E}}

\global\long\def\shF{\mathcal{F}}

\global\long\def\shG{\mathcal{G}}

\global\long\def\shH{\mathcal{H}}

\global\long\def\shI{\mathcal{I}}

\global\long\def\shJ{\mathcal{J}}

\global\long\def\shK{\mathcal{K}}

\global\long\def\shL{\mathcal{L}}

\global\long\def\shM{\mathcal{M}}

\global\long\def\shN{\mathcal{N}}

\global\long\def\shO{\mathcal{O}}

\global\long\def\shP{\mathcal{P}}

\global\long\def\shQ{\mathcal{Q}}

\global\long\def\shR{\mathcal{R}}

\global\long\def\shS{\mathcal{S}}

\global\long\def\shT{\mathcal{T}}

\global\long\def\shU{\mathcal{U}}

\global\long\def\shV{\mathcal{V}}

\global\long\def\shW{\mathcal{W}}

\global\long\def\shX{\mathcal{X}}

\global\long\def\shY{\mathcal{Y}}

\global\long\def\shZ{\mathcal{Z}}

\global\long\def\st{\ | \ }

\global\long\def\stA{\mathcal{A}}

\global\long\def\stB{\mathcal{B}}

\global\long\def\stC{\mathcal{C}}

\global\long\def\stD{\mathcal{D}}

\global\long\def\stE{\mathcal{E}}

\global\long\def\stF{\mathcal{F}}

\global\long\def\stG{\mathcal{G}}

\global\long\def\stH{\mathcal{H}}

\global\long\def\stI{\mathcal{I}}

\global\long\def\stJ{\mathcal{J}}

\global\long\def\stK{\mathcal{K}}

\global\long\def\stL{\mathcal{L}}

\global\long\def\stM{\mathcal{M}}

\global\long\def\stN{\mathcal{N}}

\global\long\def\stO{\mathcal{O}}

\global\long\def\stP{\mathcal{P}}

\global\long\def\stQ{\mathcal{Q}}

\global\long\def\stR{\mathcal{R}}

\global\long\def\stS{\mathcal{S}}

\global\long\def\stT{\mathcal{T}}

\global\long\def\stU{\mathcal{U}}

\global\long\def\stV{\mathcal{V}}

\global\long\def\stW{\mathcal{W}}

\global\long\def\stX{\mathcal{X}}

\global\long\def\stY{\mathcal{Y}}

\global\long\def\stZ{\mathcal{Z}}

\global\long\def\then{\ \Longrightarrow\ }

\global\long\def\L{\textup{L}}

\global\long\def\l{\textup{l}}

\setcounter{section}{0}
\maketitle


\section*{Introduction}

Let $K\subseteq L$ be a finite separable field extension. The extension is called Galois if all automorphisms of the algebraic closure $\bar{L}$ fixing $K$ send $L$ to $L$. In a more 
geometric language this means that $K\subseteq L$ is Galois if and only if
$\Spec L\arr \Spec K$ is a torsor under the automorphism group $\Aut_K(L)$. It
is not true that all finite separable extensions are Galois. Even worse,
it could happen that even for two successive Galois extensions $K\subseteq E$ and $E\subseteq L$ the tower $K\subseteq L$ is, although still finite separable, not Galois anymore. For example, $\Q\subseteq \Q[\sqrt{2}]\subseteq\Q[\sqrt[4]{2}]$ is a tower of Galois extensions, but itself is not Galois. Since  $L=K[\beta]$ for some $\beta\in L$, 
there is a field $L'$ inside a\textit{ chosen algebraic closure} $\bar{L}$ of
$L$, obtained from $L$ by adjoining all the roots of the minimal polynomial of $\beta$ to $L$. The extension $K\subseteq L'$ enjoys the following properties:
\begin{itemize}
    \item The extension $K\subseteq L'$ is Galois.
    \item For any Galois extension $K\subseteq M$ inside $\bar{L}$ with $M$ containing $L$, we have $L'\subseteq M$. 
\end{itemize}
The field $L'$, or more precisely the extension $K\subseteq L'$, is called the \textit{Galois closure} of $K\subseteq L$.  When
$K\subseteq L$ is a tower of two successive Galois extensions $K\subseteq E$ and $E\subseteq L$, then $L\subseteq L'$, $E\subseteq L'$ are also Galois extensions. If we denote $G\coloneqq \Aut_K(E)$, $H\coloneqq \Aut_E(L)$, then
the Galois closure $L'$ provides the following data ($\Diamond$):

\begin{itemize}     \item The automorphism group $\sG\coloneqq\Aut_K(L')$ and  group homomorphisms $\alpha\colon \sG\twoheadrightarrow G$ and $\Ker(\alpha)=\Aut_E(L')\twoheadrightarrow H$.
    \item A $\sG$-torsor, namely $\Spec L'\to \Spec K$, together with a factorization $\Spec L'\arr\Spec L$ such that $\Spec L'\to\Spec E$ is $\sG$-equivariant and $\Spec L'\arr\Spec L $ is $\Ker(\alpha)$-equivariant.
\end{itemize}

In the construction of the  \'etale fundamental group, Grothendieck completely generalized the Galois theory of fields to  that of schemes. In the world of schemes a "finite separable field extension" becomes a \textit{finite \'etale} morphism, and
a "Galois extension" becomes a \textit{connected} finite \'etale cover which is a torsor under its automorphism group, namely 
a  \textit{Galois cover}. Let $X$ be a connected scheme equipped with a geometric point $x\colon \Spec\bar{k}\arr X$, 
and let $f\colon Z\arr X$ be a connected surjective finite \'etale cover. Using Grothendieck's equivalence between the
category of finite \'etale covers of $X$ and the category of finite sets with a continuous $\pi_1^{\et}(X,x)$-action,
one can identify $f$ with a finite set with a transitive $\pi_1^\et(X,x)$-action or, after a choice of a geometric point $z$ in $Z$,
an open subgroup $\pi$ of $\pi_1^\et(X,x)$. Thus there is a pointed finite \'etale cover $f'\colon Z'\arr X$ mapping to $Z\arr X$ given by the maximal
normal subgroup of $\pi_1^{\et}(X,x)$ contained in $\pi$, and this pointed cover enjoys the following properties (\ding{73}):
\begin{itemize}
    \item The cover $f'\colon Z'\arr X$ is Galois.
    \item The pointed cover $(Z',z')$ with the map $(Z',z')\arr(Z,z)$ has the following universal property: if $(T,t)\arr(X,x)$ is
        another pointed Galois cover mapping to $(Z,z)$, then there exists a unique map $(T,t)\arr(Z',z')$ factorising $(Z',z')\arr(Z,z)$.
\end{itemize}
Note that since $f'$ corresponds to a normal subgroup of $\pi_1^\et(X,x)$ it is independent of the choice of the base points $z,x$.   The map $f'$ is the Galois closure  of $f$ (see also \cite[Proposition 5.3.9, pp. 169]{Sz} for a different approach).
If $f$ is moreover a tower of Galois covers $Z\arr Y$ and $Y\arr X$, then replacing the automorphisms of field extensions by the automorphisms of covers the Galois closure $Z'\arr X$ enjoys exactly the same properties ($\Diamond$) listed in the field case.

Now a natural question is: what about covers which may not be \'etale such as the Kummer covers in characteristic $p>0$? What would be a "Galois closure" in that case? The present paper
is set out to address these issues. As in the \'etale case we resort to the machinery of fundamental groups. Here the fundamental group is the Nori fundamental group $\pi^\NN(-,*)$ which was introduced in \cite{Nori} or its generalization - the
Nori fundamental gerbe $\Pi^\NN_-$ introduced in \cite{BV}. In Definition \ref{inflexible} we define the notion of the Nori fundamental gerbe.
Its existence is characterized in \ref{inflexible properties}, and its Tannakian property is in Theorem 
\ref{monodromy essentially finite}.  We assume the base $\sX$  to be \textit{inflexible} to guarantee the existence of the fundamental gerbe, and this "inflexible"  property plays the same role as the "connected"  property in the \'etale case. We assume $\sX$ to be \textit{pseudo-proper} (Definition \ref{pseudo-proper}) in order to use the Tannakian property. The next step is to understand
what are the "finite separable extensions" or the "finite \'etale morphisms" in
this context. For this, we introduce the notion of "essentially finite covers"
in Definition \ref{essentially finite cover}. The "Galois extensions" or the
"Galois covers" are replaced by "Nori-reduced torsors" (Definition
\ref{Nori-reduced}). Under this setting we obtain a very pretty Galois theory:

\begin{thmI}\label{main thm Nori gerbe of ess finite cover}
Let $\stX$ be a pseudo-proper and inflexible algebraic stack of finite type over $k$ and let $f\colon \stY\arr \stX$ be an
essentially finite cover. If $\car k > 0$ assume that either $f$ is \'etale or $\dim_k \Hl^1(\stX,E)<\infty$ for all vector bundles $E$ on $\stX$. Then 
\begin{enumerate}
    \item There exists a finite map $\Pi\arr \Pi^\NN_{\stX/k}$, which is unique up to equivalence,  whose base change along $\stX\arr \Pi^\NN_{\stX/k}$ is $\stY\arrdi f \stX$. Moreover,  $\stY$ is inflexible over $k$ if and only if $\Pi$ is a gerbe over $k$.
In this case $\stY\arr \Pi$ is the Nori fundamental gerbe of $\stY$, and there is a $2$-Cartesian diagram,
    \[
  \begin{tikzpicture}[xscale=1.8,yscale=-1.2]
    \node (A0_0) at (0, 0) {$\stY$};
    \node (A0_1) at (1, 0) {$\Pi^\NN_{\stY/k}$};
    \node (A1_0) at (0, 1) {$\stX$};
    \node (A1_1) at (1, 1) {$\Pi^\NN_{\stX/k}$};
    \path (A0_0) edge [->]node [auto] {$\scriptstyle{}$} (A0_1);
    \path (A1_0) edge [->]node [auto] {$\scriptstyle{}$} (A1_1);
    \path (A0_1) edge [->]node [auto] {$\scriptstyle{}$} (A1_1);
    \path (A0_0) edge [->]node [auto] {$\scriptstyle{f}$} (A1_0);
  \end{tikzpicture}
  \]
 and we have $\EF(\Vect(\stY))=\{V\in\Vect(\stY)\st f_*V\in\EF(\Vect(\stX))\}$,  where $\EF(\Vect(\sX))$ denotes the category of essentially finite vector bundles on $\sX$ (see Definition \ref{essentially finite}).
\item \textup{The \'etale case}.
If $f$ is \'etale then $\stY$ is inflexible over $k$ if and only if $\Hl^0(\stY,\odi\stY)=k$ and in this case there are $2$-Cartesian
diagrams  \[
  \begin{tikzpicture}[xscale=2,yscale=-1.2]
    \node (A0_0) at (0, 0) {$\stY$};
    \node (A0_1) at (1, 0) {$\Pi^\NN_{\stY/k}$};
    \node (A0_2) at (2, 0) {$\Pi^{\NN,\et}_{\stY/k}$};
    \node (A1_0) at (0, 1) {$\stX$};
    \node (A1_1) at (1, 1) {$\Pi^\NN_{\stX/k}$};
    \node (A1_2) at (2, 1) {$\Pi^{\NN,\et}_{\stX/k}$};
    \path (A0_0) edge [->]node [auto] {$\scriptstyle{}$} (A0_1);
    \path (A0_1) edge [->]node [auto] {$\scriptstyle{}$} (A0_2);
    \path (A1_0) edge [->]node [auto] {$\scriptstyle{}$} (A1_1);
    \path (A1_1) edge [->]node [auto] {$\scriptstyle{}$} (A1_2);
    \path (A0_2) edge [->]node [auto] {$\scriptstyle{}$} (A1_2);
    \path (A0_0) edge [->]node [auto] {$\scriptstyle{f}$} (A1_0);
    \path (A0_1) edge [->]node [auto] {$\scriptstyle{}$} (A1_1);
  \end{tikzpicture}
  \]
\item \textup{The torsor case}. 
If $f$ is a torsor under a finite group scheme $G$ over $k$ the following are equivalent: 

\begin{itemize}
\item $\stY$ is inflexible over $k$; 
\item $\Hl^0(\stY,\odi\stY)=k$; 
\item $f$ is Nori-reduced over $k$. 
\end{itemize}

Under these conditions there are $2$-Cartesian diagrams:
  \[
  \begin{tikzpicture}[xscale=2.0,yscale=-1.2]
    \node (A0_0) at (0, 0) {$\stY$};
    \node (A0_1) at (1, 0) {$\Pi^\NN_{\stY/k}$};
    \node (A0_2) at (2, 0) {$\Spec k$};
    \node (A1_0) at (0, 1) {$\stX$};
    \node (A1_1) at (1, 1) {$\Pi^\NN_{\stX/k}$};
    \node (A1_2) at (2, 1) {$\Bi G$};
    \path (A0_0) edge [->]node [auto] {$\scriptstyle{\beta}$} (A0_1);
    \path (A0_1) edge [->]node [auto] {$\scriptstyle{\pi}$} (A1_1);
    \path (A1_0) edge [->]node [auto] {$\scriptstyle{\alpha}$} (A1_1);
    \path (A1_1) edge [->>]node [auto] {$\scriptstyle{}$} (A1_2);
    \path (A0_2) edge [->]node [auto] {$\scriptstyle{}$} (A1_2);
    \path (A0_0) edge [->]node [auto] {$\scriptstyle{f}$} (A1_0);
    \path (A0_1) edge [->]node [auto] {$\scriptstyle{}$} (A0_2);
  \end{tikzpicture}
  \]\end{enumerate}
\end{thmI}

For pointed covers Theorem \ref{main thm Nori gerbe of ess finite cover}  yields the following Galois correspondence.
\begin{corI}\label{main thm Nori gerbe pointed case}
 Let $\stX$ be a pseudo-proper and inflexible algebraic stack of finite type over $k$ with a rational point $x\in \stX(k)$. If $\car k > 0$ assume that $\dim_k \Hl^1(\stX,E)<\infty$ for all vector bundles $E$ on $\stX$. Then there is an equivalence of categories
   \[
  \begin{tikzpicture}[xscale=8.2,yscale=-0.9]
    \node (A0_0) at (0, 0) {$\left\{\begin{array}[c]{c}
 \text{Pointed essentially finite covers} \\
 (\stY,y)\arr (\stX,x) \text{ with } \stY \text{ inflexible}
\end{array}
\right\}$};
    \node (A0_1) at (1, 0) {$\left\{\begin{array}[c]{c}
 \text{Subgroups } H < \pi^\NN(\stX,x)\\
 \text{of } \text{ finite index}
\end{array}
\right\}$};
    \node (A1_0) at (0, 1) {$(\stY,y)\arr(\stX,x)$};
    \node (A1_1) at (1, 1) {$\pi^\NN(\stY,y)$};
    \path (A0_0) edge [->]node [auto] {$\scriptstyle{}$} (A0_1);
    \path (A1_0) edge [|->,gray]node [auto] {$\scriptstyle{}$} (A1_1);
  \end{tikzpicture}
  \]
  where in the right-hand side we consider inclusions as arrows. Moreover
  \begin{itemize}
   \item an essentially finite cover $(\stY,y)\arr (\stX,x)$ with $\stY$ inflexible is a torsor under a finite group $G$ if and only if $\pi^\NN(\stY,y)$ is a normal subgroup of $\pi^\NN(\stX,x)$ and in this case there is an exact sequence
   \[
1\arr \pi^\NN(\stY,y)\arr\pi^\NN(\stX,x)\arr G\arr 1
\]
 \item an essentially finite cover $(\stY,y)\arr (\stX,x)$ with $\stY$ inflexible is \'etale if and only if the finite scheme $\pi^\NN(\stX,x)/\pi^\NN(\stY,y)$ is \'etale over $k$.
  \end{itemize}
\end{corI}

The above sequence was already proved to be exact in \cite[Theorem 2.9]{EHS} under the assumption that $G$ is \'etale. 
The main difficulty in the present work is to show that, if $f\colon \stY\arr \stX$ is an essentially finite cover, then $f_*$ preserves essentially finite vector bundles. The statement would be false without assuming that $\sX$ is pseudo-proper and $\dim_k \Hl^1(\stX,E)<\infty$ (see Lemmas \ref{push of ess finite is still ess finite}, \ref{push of ess finite for etale}
and, for a counter-example, Example \ref{counterexample no H1 exact sequence}).
A key tool in the proof is a characterization of essentially finite vector bundles given in \cite{TZ2} which generalizes previous results of \cite{BDS} and \cite{AM}.\\

Here is our main result on finding Galois closures for covers which may not be \'etale:

\begin{thmII}\label{closure of an essentially finite cover}
Let $\stX$ be a pseudo-proper and inflexible fibered category over $k$
and $f\colon\stY\arr\stX$ an essentially finite cover with a rational point
$y\in \stY(k)$. Denote by $\Gamma$ the monodromy gerbe of
$f_*\odi\stY$ in $\EF(\Vect(\stX))$ and by $\Delta\arr \Gamma$ the cover
in \eqref{eq1} of Lemma \ref{trivial section iff gerbe} that extends $f$. Then

\begin{enumerate}
    
    \item There is a diagram with Cartesian squares
   \[
  \begin{tikzpicture}[xscale=2.0,yscale=-1.2]
    \node (A0_0) at (0, 0) {$\Spec k$};
    \node (A0_1) at (1, 0) {$\stP_f$};
    \node (A0_2) at (2, 0) {$\Spec k$};
    \node (A1_1) at (1, 1) {$\stY$};
    \node (A1_2) at (2, 1) {$\Delta$};
    \node (A2_1) at (1, 2) {$\stX$};
    \node (A2_2) at (2, 2) {$\Gamma$};
    \path (A0_0) edge [->]node [auto] {$\scriptstyle{p}$} (A0_1);
    \path (A0_1) edge [->]node [auto] {$\scriptstyle{\lambda}$} (A1_1);
    \path (A0_1) edge [->]node [auto] {$\scriptstyle{}$} (A0_2);
    \path (A1_1) edge [->]node [auto] {$\scriptstyle{u'}$} (A1_2);
    \path (A0_2) edge [->]node [auto] {$\scriptstyle{}$} (A1_2);
    \path (A0_0) edge [->,bend left=10]node [auto] {$\scriptstyle{y}$} (A1_1);
    \path (A0_0) edge [->,bend left=20]node [auto] {$\scriptstyle{x}$} (A2_1);
    \path (A1_1) edge [->]node [auto] {$\scriptstyle{f}$} (A2_1);
    \path (A2_1) edge [->]node [auto] {$\scriptstyle{u}$} (A2_2);
    \path (A1_2) edge [->]node [auto] {$\scriptstyle{}$} (A2_2);
  \end{tikzpicture}
  \]
such that the map $\pi\colon \stP_f\arr \stX$ is a pointed Nori-reduced torsor under the
finite group scheme $G_f=\Autsh_\Gamma(u(x))$, and $\Gamma$ is the monodromy gerbe
of $\pi_*\odi{\stP_f}$. 

\item The map $\lambda\colon \stP_f\arr \stY$ is faithfully flat if and only if $\Delta$ is a gerbe (e.g. when $\sY$ is inflexible) in 
which case it is a Nori-reduced torsor for a subgroup scheme of $G_f$.

\item In the general case $\lambda \colon \stP_f \arr \stY$ factors as $\eta \circ \lambda '$, where 
$\lambda '\colon \stP_f \arr \stY '$ is faithfully flat and a Nori-reduced torsor for a finite 
subgroup scheme of $G_f$, and $\eta \colon \stY ' \arr \stY$ is a closed immersion.

\item Finally the torsor $\pi\colon \stP_f\arr \stX$ has the following universal property: for any 
pointed Nori-reduced torsor $g \colon (\stT,t) \arr (\stX,x)$ for a finite group scheme $G$, 
and any pointed faithfully flat $\stX$-morphism $h: (\stT,t) \to (\stY,y)$, there is a unique 
factorization $h=\lambda \circ j$, where $j\colon (\stT,t) \arr (\stP _f,p)$ is equivariant with 
respect to a surjective homomorphism $G\arr G_f$.\end{enumerate}
\end{thmII}
We then define the {Galois closure} of a pointed essentially finite cover
$f\colon \sY\arr\sX$ to be the Nori-reduced torsor $\sP_f\arr \sX$ in Theorem 
\ref{closure of an essentially finite cover}. In particular,  if $f\colon \sY\arr\sX$ is a pointed inflexible essentially 
finite cover (see Definition \ref{inflexible essentially finite covers}), then it admits a Galois closure in the sense of
(\ding{73}) (see Lemma \ref{GC for inflexible essentially finite covers}).

Next we deal with the Galois closure problem for towers. Let $G$ and $H$ be finite group schemes over a field $k$.
A $(G,H)$-tower of torsors over a  $k$-algebraic stack $\stX$ consists of maps $f\colon\stZ \arrdi h \stY \arrdi g \stX$ where
$h$ and $g$ are an $H$-torsor and a $G$-torsor respectively. We first show that under some conditions $f=g\circ h$ is an essentially finite
cover (Lemma \ref{towers are ess finite}). Then we define the notion of a \textit{Galois envelope} (see Definition \ref{Galois closure}) in the sense of ($\Diamond$). We investigate when the torsor $\sP_f$ obtained in
Theorem \ref{closure of an essentially finite cover} for $\sZ\arr \sX$ is a Galois envelope, if $\sP_f$ is really a Galois envelope then we call it the \textit{Galois closure  for the tower $\stZ \arrdi h \stY \arrdi g \stX$}.
Here we do not insist that $g,h$ are Nori-reduced torsors (which are "Galois covers" in the \'etale case). 
In order to obtain the Galois envelope we study the moduli $\Bi(G,H)$ of all $(G,H)$-towers  (Definition \ref{the stack B(G,H)}) which is an algebraic stack locally of finite type over $k$ (Proposition \ref{geometry of BGH}).  
Our main result for pointed towers is the following:

\begin{thmIII}\label{main theorem for Galois closure}
 Let $\stX$ be a pseudo-proper (see Definition \ref{pseudo-proper}) and inflexible
(see Definition \ref{inflexible}) algebraic stack of finite type over $k$. If $\car k > 0$ assume that 
$\dim_k \Hl^1(\stX,E)<\infty$ for all vector bundles $E$ on $\stX$. 
Then, if $\stZ\arrdi h \stY\arrdi g \stX$ is a $(G,H)$-tower of torsors then $\stZ\arr \stX$ is an essentially finite cover. 
Assume moreover that $\sZ$ has a $k$-point $z\in \stZ(k)$. Then the tower admits a pointed Nori-reduced Galois closure $(\stP,p) \arrdi \lambda (\stZ,z)$ such that:
 \begin{enumerate}
 \item $\lambda$ is faithfully flat if and only if the torsors in the tower are Nori-reduced, or equivalently $\stZ$ and $\stY$ are inflexible, and in this case it is a Nori-reduced torsor under a finite group scheme.
  \item $\stP\arr \stX$ is a torsor under a finite subgroup of the affine and of finite type $k$-group scheme $\Autsh_{\Bi(G,H)}(\xi)$, where $\xi\in \Bi(G,H)(k)$ is the tower fiber of the given tower over $x=gh(z)\colon \Spec k\arr \stX$.
 \end{enumerate}
\end{thmIII}

Previous attempts to give an affirmative answer to this question failed: both \cite{G} and
 \cite{ABE} (unpublished) contain mistakes in the proof of their main theorems. In \cite{G}, Garuti claims a functorial construction of the "Galois closure" (which is the Galois envelope in our paper) for all towers of torsors under finite
 locally free group schemes over any locally Noetherian base $B$. He works
 basically without any assumption. But this is too good to be true. Without
 rational points the claim is wrong, otherwise the map in \cite[Corollary 5.14, pp. 41]{Zh} would be an isomorphism. Example \ref{counterexample 2} shows that even for a smooth connected affine scheme $X$ over an algebraically closed field $k$ there are pointed towers without a Galois envelope (because $X$ is not pseudo-proper).  Example \ref{counterexample 1} shows that the assumption $\dim_k \Hl^1(\stX,E)<\infty$ is also important. This condition is needed to ensure that every $\mathbb G_a$-torsor over a cover of $\stX$ is induced by a torsor under a finite subgroup of $\mathbb G_a$.
Notice that the cohomological assumption is met for proper algebraic stacks of finite type over $k$ (\cite{Fal}). Moreover a geometrically connected and geometrically reduced algebraic stack over $k$ is inflexible (see Remark \ref{inflexible properties}).

Recently, in \cite{Otabe} S. Otabe used our result to study the lifting problem for linearly reductive torsors over curves. He also studied the Galois closure problem in the appendix for non-proper base schemes.

We also study a similar problem for the $S$-fundamental gerbe (see \cite{BPS}, \cite{BHD}, \cite{L1}, \cite{L2}). Let $\Ns(\sX)$ be the category of semi-stable vector bundles on $\sX$ (see Definition \ref{Nori 
    semistable}). The following analog of Theorem \ref{main thm Nori gerbe of ess finite 
cover} is proved for $S$-fundamental gerbes.

\begin{thmIV}\label{main thm for the S-gerbe}
 Let $\stX$ be a pseudo-proper algebraic stack of finite type over $k$ and with an $\SS$-fundamental gerbe. Then $\stX$ is inflexible and the profinite quotient of the $\SS$-fundamental gerbe of $\stX$ over $k$ is the Nori fundamental gerbe of $\stX$ over $k$. Let also $f\colon \stY\arr \stX$ be an essentially finite cover with $\stY$ inflexible and, if $\car k > 0$, assume that $\dim_k \Hl^1(\stX,E)<\infty$ for all vector bundles $E$ on $\stX$. Then
 $$
 \Ns(\stY)=\{ V \in \Vect(\stY) \st f_*V\in \Ns(\stX) \}
 $$
 and $\stY$ has an $\SS$-fundamental gerbe fitting in a $2$-Cartesian diagram
   \[
  \begin{tikzpicture}[xscale=1.8,yscale=-1.1]
    \node (A0_0) at (0, 0) {$\stY$};
    \node (A0_1) at (1, 0) {$\Pi^\SS_\stY$};
    \node (A0_2) at (2, 0) {$\Pi^\NN_\stY$};
    \node (A1_0) at (0, 1) {$\stX$};
    \node (A1_1) at (1, 1) {$\Pi^\SS_\stX$};
    \node (A1_2) at (2, 1) {$\Pi^\NN_\stX$};
    \path (A0_0) edge [->]node [auto] {$\scriptstyle{}$} (A0_1);
    \path (A0_1) edge [->]node [auto] {$\scriptstyle{}$} (A0_2);
    \path (A1_0) edge [->]node [auto] {$\scriptstyle{}$} (A1_1);
    \path (A0_2) edge [->]node [auto] {$\scriptstyle{}$} (A1_2);
    \path (A1_1) edge [->]node [auto] {$\scriptstyle{}$} (A1_2);
    \path (A0_0) edge [->]node [auto] {$\scriptstyle{}$} (A1_0);
    \path (A0_1) edge [->]node [auto] {$\scriptstyle{}$} (A1_1);
  \end{tikzpicture}
  \]
In particular if $y\in \stY(k)$ and $x=f(y)\in \stX(k)$ then there is a Cartesian diagram of affine group schemes
    \[
  \begin{tikzpicture}[xscale=2.8,yscale=-1.1]
    \node (A0_0) at (0, 0) {$\pi^\SS(\stY,y)$};
    \node (A0_1) at (1, 0) {$\pi^\NN(\stY,y)$};
    \node (A1_0) at (0, 1) {$\pi^\SS(\stX,x)$};
    \node (A1_1) at (1, 1) {$\pi^\NN(\stX,x)$};
    \path (A0_0) edge [->>]node [auto] {$\scriptstyle{}$} (A0_1);
    \path (A1_0) edge [->>]node [auto] {$\scriptstyle{}$} (A1_1);
    \path (A0_1) edge [right hook->]node [auto] {$\scriptstyle{}$} (A1_1);
    \path (A0_0) edge [right hook->]node [auto] {$\scriptstyle{}$} (A1_0);
  \end{tikzpicture}
  \]
  When $f: \stY \arr \stX$ is a pointed torsor under a finite group scheme $G$, then the following sequence is exact
  $$1 \arr \pi^\SS(\stY,y) \arr \pi^\SS(\stX,x) \arr G \arr 1. $$
\end{thmIV}

The paper is divided as follows. In the first section we recall part of the machinery about Nori fundamental gerbes and prove some preliminary results. In the second section we study essentially finite covers and their Nori fundamental gerbes, proving in particular Theorem \ref{main thm Nori gerbe of ess finite 
cover}, Theorem \ref{closure of an essentially finite cover} and Corollary \ref{main thm Nori gerbe pointed case}. In the third section we study towers of torsors and their Galois closures, proving Theorems \ref{Galois closure when G or H is etale} and \ref{main theorem for Galois closure}, while in the fourth section we study the $S$-fundamental gerbe and prove Theorem \ref{main thm for the S-gerbe}. Finally in the last section we collect some counter-examples.

\section*{Acknowledgement}
 We would like to thank  H\'el\`ene Esnault and Angelo Vistoli
for helpful conversations and suggestions received. We would also like to thank
the referee for many helpful remarks.
\section{Notation and Preliminaries}

\subsection{Notation}

By a fibered category over a scheme $S$ we will always mean a category fibered in 
groupoids over the category $\Aff/S$ of affine schemes over $S$.

Let $\sX$ be a fibered category over  $\Aff/S$, where $S$ is an affine scheme.
An fpqc-altlas from a scheme $U$ is a map $U\arr \sX$ representable by fpqc-coverings of algebraic spaces, i.e. for any algebraic space $T$ mapping to $\sX$ the fibered product $U\times_\sX T$ is again an algebraic space over $S$ which is an fpqc covering of $T$. A fibered category $\sX$ is said to be quasi-compact if it has an fpqc-atlas from a quasi-compact scheme. If $f\colon \sX\arr\sY$ is a quasi-compact and quasi-separated map between fibered categories admitting fpqc-atlases from schemes, then the pullback $f^*:\QCoh(\sY)\arr\QCoh(\sX)$ admits a right adjoint $f_*:\QCoh(\sX)\arr\QCoh(\sY)$ which commutes with flat base change, where $\QCoh(-)$ denotes the category of quasi-coherent sheaves (see \cite[Proposition 1.7]{Ton}).

A cover of a fibered category $\stX$ is a finite, flat and finitely presented
morphism or, equivalently, an affine map $f\colon \stY\arr \stX$ with the
property that
$f_*\odi\stY$ is locally free of finite rank. If $\stX$ is defined over a field
$k$ a pointed cover over $k$ $$(\stY,y)\arr(\stX,x)$$ is a cover $f\colon \stY\arr
\stX$ with $x\in \stX(k)$ and $y \in \stY_x(k)$, where $\stY_x$ denotes the fiber
of $f$ over $x$, which is a finite $k$-scheme; equivalently $y\,\in \,\stY(k)$ with
a given isomorphism $f(y) \,\simeq\, x$.

Given a morphism of schemes $U\arr V$ and a functor $F\colon \Aff/U \arr \sets$, 
the Weil restriction of $F$ along $U \arr V$ is the functor
\[
    W_{U/V}(F)\colon \Aff/V \arr \sets\comma\ Z\,\longmapsto\, \Hom(Z\times_V U,F)\, .
\]
Given any functor $G\colon \Aff/V \arr \sets$, we set $W_{U}(G)\coloneqq W_{U/V}(G\times_V U)$.

Injectivity and surjectivity of morphisms of group schemes always mean the corresponding properties for fpqc sheaves. For affine group schemes over a field an injective morphism is a closed immersion and a surjective morphism is faithfully flat 
(\cite[Theorem 15.5]{Wat} for surjectivity).

\subsection{Preliminaries}

We will now recall some results used in later sections.

Fix a base field $k$.

For properties of affine gerbes over a field (often improperly called just
gerbes) and Tannakian categories used here the reader is referred to
\cite[Appendix B]{TZ}.

 \begin{defn}[{\cite[Definition 5.1, Definition 5.3]{BV}}]\label{inflexible}
For a fibered category $\stX$ over $k$, the \emph{Nori fundamental gerbe} 
(respectively, \emph{Nori \'etale fundamental gerbe}) of $\stX/k$ is a profinite
(respectively, pro\'etale) gerbe $\Pi$ over $k$ together with a map $\stX\arr \Pi$
such that for all finite (respectively, finite and \'etale) stacks $\Gamma$ over $k$
the pullback functor
 \[
 \Hom_k(\Pi,\Gamma)\arr \Hom_k(\stX,\Gamma)
 \]
 is an equivalence of categories. Furthermore, if this gerbe exists, it is unique up to a unique
isomorphism and in that case it will be denoted by $\Pi^\NN_{\stX/k}$ (respectively,
$\Pi^{\NN,\et}_{\stX/k})$; sometimes $/k$ will be dropped if it is clear from the context.

We call $\sX$ \emph{inflexible} if it is non-empty and all maps from it to a finite stack
over $k$ factor through a finite gerbe over $k$.
\end{defn}

\begin{rmk}\label{inflexible properties}
By \cite[p.~13, Theorem 5.7]{BV} $\sX$ admits a Nori fundamental gerbe if and only if
it is inflexible; in this case, the Nori \'etale fundamental gerbe of $\stX$ is the
maximal pro\'etale quotient of the Nori fundamental gerbe of $\stX$.
If $\stX$ is reduced, quasi-compact and quasi-separated,
then $\stX$ is inflexible if and only if $k$ is algebraically closed in
$\Hl^0(\odi \stX)$ \cite[Theorem 4.4]{TZ}. In particular if $\stX$ is geometrically
connected and geometrically reduced, then it is inflexible.
\end{rmk}

\begin{defn}
 If $\stX$ is an inflexible fibered category over $k$ with a rational point
$x\in \stX(k)$ and Nori gerbe $\psi\colon \stX\arr \Pi^\NN_{\stX/k}$,
the \emph{Nori fundamental group scheme} $\pi^\NN(\stX/k,x)$ of $(\stX,x)$ over
$k$ is the sheaf of automorphisms of $\psi(x)\in \Pi^\NN_{\stX/k}(k)$. Again
$/k$ will often be dropped if it is clear from the context. 
\end{defn}

\begin{rmk}
The Nori fundamental group scheme $\pi^\NN(\stX,x)$ is a profinite group scheme and its classifying stack
$\Bi\pi^\NN(\stX,x)$ is isomorphic to $\Pi^\NN_\stX$ (the trivial torsor is sent to $\psi(x)$). The
universal property of $\Pi^\NN_\stX$ translates into the following: for all finite group schemes
$G$ over $k$ the map 
  \[
  \begin{tikzpicture}[xscale=8.1,yscale=-0.7]
    \node (A0_0) at (0, 0) {$\Hom_{k\text{-groups}}(\pi^\NN(\stX,x),G)$};
    \node (A0_1) at (1, 0) {$\{\text{pointed }G\text{-torsors }(\stP,p)\arr (\stX,x)\}/\simeq$};
    \node (A1_0) at (0, 1) {$(\pi^\NN(\stX,x)\arr  G)$};
    \node (A1_1) at (1, 1) {$(\stX\arr \Bi \pi^\NN(\stX,x)\arr \Bi G)$};
    \path (A0_0) edge [->]node [auto] {$\scriptstyle{}$} (A0_1);
    \path (A1_0) edge [|->,gray]node [auto] {$\scriptstyle{}$} (A1_1);
  \end{tikzpicture}
  \]
is bijective.
\end{rmk}

\begin{defn}[{\cite[p.~21, Definition 7.7]{BV}}]\label{essentially finite}
Let $\shC$ be an additive and monoidal category. An object $E\in\shC$ is called \emph{finite}
if there exist polynomials $f\neq g \,\in\, \N[X]$ and an isomorphism $f(E)\simeq g(E)$;
the object $E$ is called \emph{essentially finite} if it is a kernel of a homomorphism of
finite objects of $\shC$. Let $\EF(\shC)$ denote the full subcategory of $\shC$ consisting of essentially finite objects.
\end{defn}

\begin{defn}[{\cite[p.~20, Definition 7.1]{BV}}] \label{pseudo-proper}
A category $\sX$ fibered in groupoids over a field $k$ is \textit{pseudo-proper} if it satisfies the following two conditions:
\begin{enumerate}
\item there exists a quasi-compact scheme $U$ together with a morphism $U\arr \stX$ which is
representable, faithfully flat, quasi-compact, and quasi-separated, and
\item for all vector bundles $E$ on $\stX$ the $k$-vector space $\Hl^0(\sX,E)$ is finite-dimensional.
\end{enumerate}
\end{defn}

\begin{ex}[{\cite[p.~20, Example 7.2]{BV}}] 
    Examples of pseudo-proper fibered categories are proper algebraic stacks and affine gerbes.
\end{ex}

\begin{rmk}\label{infl then HO=k}
 Let $\stX$ be a pseudo-proper algebraic stack of finite type over $k$. If $\stX$ is
inflexible then $\Hl^0(\odi\stX)=k$ (see \cite[Lemma 7.4]{BV}), while the converse
holds if $\stX$ is reduced (see Remark \ref{inflexible properties}).
\end{rmk}

\begin{defn} Let $\sC$ be a Tannakian category, and let $M\in\sC$ be an object. The \textit{monodromy gerbe} of $M$ is the gerbe corresponding
    to the Tannakian subcategory of $\sC$ generated by $M$. (See \cite[Definition B.8, pp. 42]{TZ})
\end{defn}

\begin{thm}[{\cite[p.~22, Theorem 7.9, Corollary 7.10]{BV}}]\label{monodromy essentially finite}
    Let $\sX$ be an inflexible pseudo-proper fibered category over a field $k$. Then the
pullback along $\stX\arr\Pi^\NN_{\stX/k}$ induces an equivalence
of categories $\Vect(\Pi_{\sX/k}^\NN)\arr \EF(\Vect(\sX))$.

Let $\sC$ be a Tannakian category. Then $\EF(\sC)$ is the Tannakian subcategory of $\shC$ of objects whose monodromy gerbe is finite.
\end{thm}

\begin{defn}\label{essentially finite cover} A cover $f \colon \stY \arr \stX$ is {\it essentially
finite} if $f_* \odi\stY$ is an essentially finite vector bundle.
\end{defn}

\begin{defn}\label{monodromy gerbe}
    Let $\stX$ be an inflexible and pseudo-proper fibered category over a field $k$. Given an object
$V$ of $\EF(\Vect(\stX))$, the gerbe corresponding to the full Tannakian subcategory of
$\EF(\Vect(\stX))$ generated by $V$ will be called the {\it monodromy gerbe} of $V$. When
$f \colon \stY \arr \stX$ is an essentially finite cover, the monodromy gerbe of the cover
is by definition the monodromy gerbe of $f_* \stO_\stY$.
\end{defn}
\begin{defn}[{\cite[Definition 5.10]{BV}}]\label{Nori-reduced}
A map $\stX\arr \Gamma$ from a fibered category over $k$ to a finite gerbe over
$k$ is called \emph{Nori-reduced} over $k$ if any faithful morphism
$\Gamma'\arrdi \alpha \Gamma$ that fits in a factorization
$\stX\arr \Gamma'\arrdi \alpha \Gamma$, where $\Gamma'$ is a gerbe, is an isomorphism.
 
 A torsor $\stP\arr \stX$ under a finite group scheme $G$ over $k$ is called Nori-reduced over $k$ if the map $\stX\arr \Bi G$ is Nori-reduced over $k$.
\end{defn}

\begin{rmk}
 If $\stX$ is inflexible, then any map from $\stX$ to a finite gerbe factors uniquely
through a Nori-reduced map (see \cite[Lemma 5.12]{BV}). Moreover $\Pi^\NN_\stX$ can
be seen as the projective limit of the Nori-reduced maps $\stX\arr \Gamma$ (see
\cite[Theorem 5.7]{BV} and its proof).
 
If $\stX$ is an inflexible and pseudo-proper fibered category, and $\phi\colon \stX\arr 
\Gamma$ is a map to a finite gerbe, then $\phi$ is Nori-reduced if and only if the induced 
map $\Pi^\NN_\stX\arr \Gamma$ is a quotient (\cite[Definition B.1, pp. 40]{TZ}); in this case $\Vect(\Gamma)\arr 
\EF(\Vect(\stX))$ is a Tannakian subcategory. This is a direct consequence of Theorem \ref{monodromy 
essentially finite} and the universal property of $\Pi^\NN_\stX$. Moreover 
$\phi_*\odi\stX\simeq \odi\Gamma$ (see \cite[Lemma 7.11]{BV}).
\end{rmk}

One of the key ingredients in the paper is the following result.

\begin{thm}[{\cite[Corollary I]{TZ2}}]\label{key thm}
Let $\stX$ be a pseudo-proper and inflexible algebraic stack of finite type over a field $k$ of
positive characteristic, and let $f\colon \stY\arr \stX$ be a surjective cover. If
$V\in \Vect(\stX)$, and $f^*V$ is free, then $V$ is essentially finite in $\Vect(\stX)$.
\end{thm}

\begin{rmk}\label{key thm for torsors}
 If $\stX$ is a pseudo-proper and inflexible algebraic stack of finite type over a
field $k$ (the characteristic is allowed to be $0$), $f\colon \stY\arr \stX$ is a surjective
\'etale cover and $V\in \Vect(\stX)$ is trivialized by $f$, then it follows that $V$ is
essentially finite with \'etale monodromy gerbe in $\EF(\Vect(\stX))$. Indeed, replacing $f$ by a Galois closure one can assume that $f$ is an \'etale Galois cover. This case is
exactly \cite[Lemma 1.4]{TZ3}.
\end{rmk}

\begin{lem}\label{key product gerbes}
Let $T''\arrdi a T$ and $T'\arrdi b T$ be two maps of affine group schemes over $k$, and let 
  \[
  \begin{tikzpicture}[xscale=1.5,yscale=-1.2]
    \node (A0_0) at (0, 0) {$\stR$};
    \node (A0_1) at (1, 0) {$\Bi T'$};
    \node (A1_0) at (0, 1) {$\Bi T''$};
    \node (A1_1) at (1, 1) {$\Bi T$};
    \path (A0_0) edge [->]node [auto] {$\scriptstyle{}$} (A0_1);
    \path (A1_0) edge [->]node [auto] {$\scriptstyle{}$} (A1_1);
    \path (A0_1) edge [->]node [auto] {$\scriptstyle{}$} (A1_1);
    \path (A0_0) edge [->]node [auto] {$\scriptstyle{}$} (A1_0);
  \end{tikzpicture}
  \]
be the corresponding $2$-Cartesian diagram. Then the following two hold.
\begin{enumerate}
\item The functor $\Psi\colon \Bi(T''\times_T T')\arr \stR$ mapping a $T''\times_T T'$-torsor to the associated
    $T''$- and $T'$- torsors is fully faithful and it is an equivalence if and only if the map 
    $T'\times T''\arr T$, $(t',t'')\mapsto b(t')a(t'')$ is an fpqc epimorphism (e.g. if $a$ or $b$ is surjective).
    In this case a quasi-inverse is obtained by mapping an object of $\shR$ given by torsors $P'',P',P$ under $T'',T',T$ respectively and equivariant maps $P'\arr P$ and $P''\arr P$ to the fiber product $P''\times_P P'$.

 \item If $T''=\Spec k$, so that $\Bi T''=\Spec k$, and $T'\arr T$ is injective, then $\stR
=T/T'$, where $T/T'\arr \Bi T'$ is induced by the $T'$-torsor $T\arr T/T'$. In particular, if
$T'$ is a finite subgroup of $T$, then $\Bi T'\arr \Bi T$ is an affine map.
\end{enumerate}
\end{lem}
\begin{proof}
The functor $\Psi$ maps the trivial torsor to $(T'',T',\id)\in \shR(k)$. A direct computation shows that the sheaf of automorphisms of this object is exactly $T''\times_T T'$ (via $\Psi$). This means that $\Psi$ is an equivalence onto the full-substack $\shR'$ of $\shR$ of objects locally isomorphic to $(T'',T',\id)$.
Thus we have to understand when $\shR'=\shR$. All objects of $\shR$ are locally isomorphic to an object of the form $(T'',T',c)\in \shR(U)$ where $U$ is an affine scheme and $c\in T(U)$ is thought of as multiplication on the left $T\arr T$. An isomorphism $(T'',T',1)\arr (T'',T',c)$ is given by $t''\in T''(U)$ and $t'\in T'(U)$ such that $ca(t'')=b(t')$. Thus $(T'',T',c)$ is locally isomorphic to $(T'',T',1)$ if and only if $c$ is in the (fpqc) image of $T'\times T''\arr T$.
The last claim of $(1)$ follows because if $\overline P$ is a $T''\times_T T'$-torsor inducing torsors $P'',P',P$ under $T'',T',T$ respectively then the commutative diagram
  \[
  \begin{tikzpicture}[xscale=1.5,yscale=-1.2]
    \node (A0_0) at (0, 0) {$\overline P$};
    \node (A0_1) at (1, 0) {$P'$};
    \node (A1_0) at (0, 1) {$P''$};
    \node (A1_1) at (1, 1) {$P$};
    \path (A0_0) edge [->]node [auto] {$\scriptstyle{}$} (A0_1);
    \path (A0_0) edge [->]node [auto] {$\scriptstyle{}$} (A1_0);
    \path (A0_1) edge [->]node [auto] {$\scriptstyle{}$} (A1_1);
    \path (A1_0) edge [->]node [auto] {$\scriptstyle{}$} (A1_1);
  \end{tikzpicture}
  \]
is automatically Cartesian: locally, after choosing a section of $\overline P$, the above diagram is the one yielding $T''\times_T T'$. Notice that the $T''\times_T T'$-space given in the last part of $(1)$ is not a torsor in general because it may fail to have sections locally.

For $(2)$, $\shR$ is the sheaf of $T'$-torsors $P$ together with an equivariant map $P\arr T$, which is represented by $T/T'$.
\end{proof}

\begin{rmk}
 If $f\colon \stY\arr \stX$ is a cover of algebraic stacks then $f_*$ preserves vector bundles.
Moreover since $f_*$ is exact we have
 \[
 \Hl^i(\stY,E)=\Hl^i(\stX,f_*E) \text{ for all } i\geq 0\comma E\in \QCoh(\stY)\, .
 \]
 In particular, if $\stX$ is pseudo-proper over $k$, then $\stY$ is pseudo-proper over
$k$. Moreover we will often use also the following property: if for all vector bundles $E$ on
$\stX$ one has $\dim_k \Hl^1(E)<\infty$ then the same holds for vector bundles on $\stY$.
\end{rmk}

\begin{lem}[{\cite[Lemma 2.5]{TZ2}}]\label{extensions of O}
 Let $\stX$ be an algebraic stack over a field $k$, of positive characteristic, such
that $\dim_k \Hl^1(X,E) <\infty$ for all vector bundles $E$ on $\stX$. Let
 \[
 \shG_0 \arr \shG_1 \arr \cdots \arr \shG_{N-1} \arr \shG_N=0
 \]
 be a sequence of surjective maps of quasi-coherent sheaves on $\stX$ such that
$\Ker(\shG_{l-1} \arr \shG_l)$ is free of finite rank for all $1\leq l \leq N$. Then there exists
a surjective cover $f\colon \stX'\arr \stX$ such that $f^*\shG_l$ is free of finite rank for
all $l$.
\end{lem}

\begin{rmk}[{\cite[Example 1.5, Corollary 1.7]{TZ}}]\label{Tannakian reconstruction}
Let $\Gamma$ be a finite stack or an affine gerbe over $k$. For all fibered categories
$\stZ$ over $k$, the pullback of vector bundles establishes an equivalence
of categories between $\Hom_k(\stZ,\Gamma)$ and the groupoid of functors
$\Vect(\Gamma)\arr \Vect(\stZ)$ which are $k$-linear, monoidal and preserves short
exact sequences in the category of quasi-coherent sheaves. 
\end{rmk}

\begin{rmk}\label{pullback from finite stack}
Let us comment on the relationship between the essentially finite vector bundles on a
fibered category $\stX$ over
$k$ and the vector bundles which are pullbacks from a finite stack.
When $\stX$ is inflexible and pseudo-proper, these two notions agree as a consequence of \cite{BV}. 
One of the key observations in \cite{BV} is that if $\Gamma$ is a finite stack over
$k$ and $V\in \Vect(\Gamma)$, then $V$ is essentially finite. More precisely, there is a finite vector bundle $E$ on $\Gamma$ and an exact sequence in $\Coh(\Gamma)$
\[
0\arr V \arr E^{\oplus a} \arr E^{\oplus b} \arr E' \arr 0
\]
 for some $a,b\in \N$ and $E'\in \Vect(\Gamma)$. In particular, if $\phi\colon
\stX\arr \Gamma$ is any map from a fibered category, then $\phi^*V$ is an essentially finite vector bundle on $\stX$.
The proof of this fact is the same as that of \cite[Lemma 7.15]{BV} together
with the following clarification. First we can assume that $\Gamma$ is connected. Let
$\rho\colon T\arr \Gamma$ be a surjective cover from a finite connected $k$-scheme $T$.
The direct image $E\,=\,\rho_*\odi T$ is finite by \cite[Lemma 7.15]{BV}. Since the
cokernel of $V\arr \rho_*\rho^* V\simeq E^{\oplus \rk V}$ is a vector bundle one can
easily construct the above sequence.
 
Now let $\stX$ be a fibered category and $V\in \Vect(\stX)$. If $V$ 
is essentially finite, one might argue that there is a homomorphism between two finite vector
bundles $q\colon E_1\arr E_2$ whose kernel is $V$. This is actually misleading.
Since the definition of essentially finite in the category $\Vect(\stX)$ is intrinsic to this 
category, $V$ has to be a kernel of $q$ inside the category $\Vect(\stX)$. This 
does not imply that $V$ coincides with the kernel $\shK$ of $q$ in $\QCoh(\stX)$. 
This equality holds if $\stX$ is pseudo-proper and inflexible. If $\stX$ has the resolution property, that is all 
quasi-coherent sheaves are quotients of sums of vector bundles (e.g.  when $\sX$ is a
quasi-projective scheme or a smooth separated scheme), and if the kernel $V$ exists in the category  $\Vect(\stX)$, then $V=\shK$.
 
In order to avoid the above-mentioned issue, if $\stX$ is pseudo-proper but not inflexible it seems to us that the 
``correct'' essentially finite vector bundles to use are vector bundles coming from a finite 
stack or at least that are kernel in $\QCoh(\stX)$ of a map of finite vector bundles. 
Although this is not an intrinsic notion it would be a good working definition. This should 
also explain why Lemma \ref{push of ess finite for etale} and Lemma \ref{push of ess finite is still ess 
finite} should be understood as results assuring that pushforward preserves essentially 
finite vector bundles. In any case in the present paper we consider essentially 
finite vector bundles only on pseudo-proper and inflexible fibered categories, so we 
maintain the notion of essentially finite in Definition \ref{essentially finite}.
 
There is a partial converse to the fact that vector bundles coming from finite stacks are essentially finite. If $\stX$ is a fibered category over $k$ with $\dim_k \Hl^0(\odi\stX)<\infty$ and $V\in \Vect(\stX)$ is a finite vector bundle, then
there exist a map $\phi\colon \stX\arr \Phi$ to a finite stack and $W\in \Vect(\Phi)$
such that $V\simeq \phi^*W$. This is essentially proved in \cite[p.~23]{BV} (just
after the proof of Lemma 7.11). We recall here the construction for the convenience of the reader. We can assume that $V$ has rank $r$ and take $f\neq g\in \N[x]$ such that $f(V)\simeq g(V)$. The group $\GL_r$ acts on the scheme $I=\Isosh(f(k^r),g(k^r))=\GL_N$ with $N=f(r)=g(r)$. The isomorphism $f(V)\simeq g(V)$ gives a factorization of the vector bundle $V\colon \stX\arr\Bi \GL_r$ through $[I/\GL_r]$ and we have Cartesian diagrams
   \[
  \begin{tikzpicture}[xscale=2.5,yscale=-1.2]
    \node (A0_1) at (1, 0) {$\Omega$};
    \node (A0_2) at (2, 0) {$I$};
    \node (A1_0) at (0, 1) {$\stX$};
    \node (A1_1) at (1, 1) {$[\Omega/\GL_r]$};
    \node (A1_2) at (2, 1) {$[I/\GL_r]$};
    \node (A1_3) at (3, 1) {$\Bi GL_r$};
    \node (A2_1) at (1, 2) {$\Spec H^0(\odi\stX)$};
    \node (A2_2) at (2, 2) {$I/GL_r$};
    \path (A1_0) edge [->]node [auto] {$\scriptstyle{}$} (A2_1);
    \path (A2_1) edge [->]node [auto] {$\scriptstyle{}$} (A2_2);
    \path (A0_1) edge [->]node [auto] {$\scriptstyle{}$} (A1_1);
    \path (A1_0) edge [->]node [auto] {$\scriptstyle{}$} (A1_1);
    \path (A1_1) edge [->]node [auto] {$\scriptstyle{}$} (A1_2);
    \path (A0_2) edge [->]node [auto] {$\scriptstyle{}$} (A1_2);
    \path (A1_1) edge [->]node [auto] {$\scriptstyle{}$} (A2_1);
    \path (A1_2) edge [->]node [auto] {$\scriptstyle{}$} (A1_3);
    \path (A0_1) edge [->]node [auto] {$\scriptstyle{}$} (A0_2);
    \path (A1_2) edge [->]node [auto] {$\scriptstyle{}$} (A2_2);
  \end{tikzpicture}
  \]
Here we are using that $I\arr 
I/\GL_r$ is a geometric quotient and $I/\GL_r$ is affine
because $\GL_r$ is geometrically reductive. Thus we must show that 
$\Phi\,=\,[\Omega/\GL_r]$ is a finite stack. As the geometric fibers of $I\arr 
I/\GL_r$ consist (topologically) of one orbit and $\Hl^0(\odi\stX)$ is a finite 
$k$-algebra one sees that $\Phi(\overline k)$ has finitely many isomorphism 
classes. The action of $\GL_r$ on $I$ has finite stabilizers by \cite[Lemma 
7.12]{BV} and hence it follows that the diagonal of $\Phi$ is quasi-finite. By 
\cite[Proposition 4.2]{BV} it follows that $\Phi$ is a finite stack.
\end{rmk}

\begin{lem}\label{fully faithful on Vect means pushforwars O is O}
    Let $\stX$ be a fibered category which admits an fpqc-atlas from a scheme,  and let
$u\colon \stX\arr \Gamma$ be a quasi-compact and quasi-separated map, where $\Gamma$ is an affine gerbe. Then $u^*\colon \Vect(\Gamma) \arr \Vect(\stX)$ is fully faithful if and only if $u^\#\colon \odi\Gamma\arr u_*\odi\stX$ is an isomorphism.
\end{lem}

\begin{proof}
For $V,V'\,\in\, \Vect(\Gamma)$ the composition 
 \[
 \Hom_\Gamma(V,V')\arrdi {u^*} \Hom_\stX(u^*V,u^*V') \simeq \Hom_\Gamma(V,u_*u^*V') \simeq \Hom_\Gamma(V,V'\otimes u_*\odi\stX)
 \]
 is the map induced by $u^\#\colon \odi\Gamma\arr u_*\odi\stX$. In particular, if
this map is an isomorphism then $u^*\colon \Vect(\Gamma)\arr \Vect(\stX)$ is fully
faithful.

Conversely, setting $V'\,=\,\odi\Gamma$ above, we have that the homomorphism $$\Hom_\Gamma
(V,\odi\Gamma)\arr \Hom_\Gamma (V,u_*\odi\stX)$$ induced by $u^\#$ is an isomorphism
for all vector bundles $V$ over $\Gamma$. Since $u_*\odi\stX$ is a quasi-coherent
sheaf, there is a surjective map $a\colon\bigoplus V_i\twoheadrightarrow u_*\odi\stX$, where $V_i$ are vector bundles on $\Gamma$ \cite[p.~132, Corollary 3.9]{De}.
The above isomorphism produces a map $\bigoplus V_i\arr\sO_\Gamma$ whose composition with $u^\#$ is $a$, so we get that $u^\#$ is surjective. Since $u\colon
\stX\arr \Gamma$ is faithfully flat we also have that $u^\#$ injective.
\end{proof}

\begin{rmk}\label{push of O Nori gerbe}
 If $\stX$ is a pseudo-proper and inflexible algebraic stack over $k$ and if
$\alpha\colon \stX\arr \Pi^\NN_\stX$ 
is the structure map of the Nori fundamental gerbe, then using Theorem \ref{monodromy essentially finite} and
Lemma \ref{fully faithful on Vect means pushforwars O is O} we conclude that
$\alpha_*\odi\stX\simeq \odi{\Pi^\NN_\stX}$. The same holds for the Nori
\'etale fundamental gerbe.  
\end{rmk}

\section{Essentially finite covers and their Nori gerbes}

Let $k$ be a base field. In this section we study the notion of an essentially finite 
cover, which generalizes the notion of torsor under a finite group scheme. 
Moreover we are going to prove Theorem \ref{main thm Nori gerbe of ess finite 
cover} and Corollary \ref{main thm Nori gerbe pointed case}.

Recall that an \emph{essentially finite} cover $f\colon \stY \arr \stX$ of fibered 
categories is a cover such that $f_*\odi{\stY}$ is essentially finite as an object of 
$\Vect(\stX)$ (see Definition \ref{essentially finite cover}).

First observe that a torsor under a finite group scheme is an essentially finite cover.
Indeed, let $f\colon \stY\arr \stX$ be a torsor under a finite group scheme $G$ over
$k$ corresponding to $u\colon \stX\arr \Bi G$. Then $u^*(k[G])\simeq f_*\odi\stY$,
where $k[G]$ is the regular representation. Applying
\cite[Lemma 7.15]{BV} to $\Spec k \arr \Bi G$ we see that $k[G]$ is finite in
$\Vect(\Bi G)=\Rep G$ and thus $f_*\odi\stY$ is a finite vector bundle.

\begin{prop}\label{equiv for essentially finite covers}
Let $\stX$ be a pseudo-proper and inflexible fibered category over $k$. Then there
is an equivalence of categories
  \[
  \begin{tikzpicture}[xscale=4.0,yscale=-1.2]
    \node (A1_0) at (0, 1) {$\left\{\begin{array}[c]{c}
 \text{Stacks finite} \\
 \text{over }\Pi^\NN_\stX
\end{array}
\right\}$};
    \node (A1_1) at (1.2, 1) {$\left\{\begin{array}[c]{c}
 \text{Essentially finite} \\
 \text{covers of }\stX
\end{array}
\right\}$};
        \path (A1_0) edge [->]node [auto] {$\scriptstyle{\Phi}$} (A1_1);
    
  \end{tikzpicture}
  \]
where $\Phi$ is the pullback along $\stX\arr \Pi^\NN_\stX$. If moreover $x\in \stX(k)$ is a rational point, then there is an equivalence
\[
  \begin{tikzpicture}[xscale=4.0,yscale=-1.2]
        \node (A1_1) at (1.2, 1) {$\left\{\begin{array}[c]{c}
 \text{Essentially finite} \\
 \text{covers of }\stX
\end{array}
\right\}$};
    \node (A1_2) at (2.6, 1) {$\left\{\begin{array}[c]{c}
 \text{Finite }k\text{-schemes with} \\
 \text{an action of }\pi^\NN(\stX,x)
\end{array}
\right\}$};
        \path (A1_1) edge [->]node [auto] {$\scriptstyle{\Psi}$} (A1_2);
  \end{tikzpicture}
  \]
where $\Psi$ is the pullback along $\Spec k\arrdi x \stX$. Furthermore, $\Psi$ extends the correspondence between pointed Nori-reduced torsors of $\stX$
and quotient group schemes of $\pi^\NN(\stX,x)$.
\end{prop}

\begin{proof}
 The functor $\Phi$ is the equivalence mapping ring objects of $\Vect(\Pi^\NN_{\stX})$
to ring objects of $\EF(\Vect(X))$. If $x\in \stX(k)$, then $\Pi^\NN_\stX
\,=\,\Bi \pi^\NN(\stX,x)$, and the ring objects of $\Vect(\Pi^\NN_\stX)\,=\,
\Rep \pi^\NN(\stX,x)$ are precisely the finite $k$-algebras with an action of
$\pi^\NN(\stX,x)$. This easily implies that $\Psi\circ\Phi$ is an equivalence.
The last claim follows by construction.
\end{proof}

\begin{lem}\label{trivial section iff gerbe}
 Let $\stX$ be a pseudo-proper and inflexible fibered category over $k$
and let $f\colon \stY\arr \stX$ an essentially finite cover; let
$u\colon\stX\arr \Gamma$ be the monodromy gerbe of $f_*\odi\stY\,\in\,\EF(\Vect(\stX))$.
Then there exist unique covers $\Delta\arr \Gamma$ and $\Pi\arr\Pi_{\sX/k}^\NN$ which make the following diagrams
   \[
  \begin{tikzpicture}[xscale=1.8,yscale=-1.2]
    \node (A0_0) at (0, 0) {$\stY$};
    \node (A0_1) at (1, 0) {$\Pi$};
    \node (A0_2) at (2, 0) {$\Delta$};
    \node (A1_0) at (0, 1) {$\stX$};
    \node (A1_1) at (1, 1) {$\Pi_{\sX/k}^\NN$};
    \node (A1_2) at (2,1) {$\Gamma$};
    \path (A0_0) edge [->]node [below] {$\scriptstyle{\beta}$} (A0_1);

\path (A0_0) edge [->, bend right=45] node [auto] {$\scriptstyle{v}$} (A0_2);

\path (A1_0) edge [->, bend left=45] node [below] {$\scriptstyle{u}$} (A1_2);

\path (A0_1) edge [->]node [auto] {$\scriptstyle{}$} (A0_2);
\path (A1_1) edge [->]node [auto] {$\scriptstyle{}$} (A1_2);
\path (A0_2) edge [->]node [auto] {$\scriptstyle{}$} (A1_2);
    \path (A0_0) edge [->]node [left] {$\scriptstyle{f}$} (A1_0);
    \path (A0_1) edge [->]node [auto] {$\scriptstyle{}$} (A1_1);
    \path (A1_0) edge [->]node [auto] {$\scriptstyle{\alpha}$} (A1_1);
  \end{tikzpicture}
  \]
Cartesian, and we also have $\beta_*\odi\stY\simeq \odi\Pi$, $v_*\odi\stY\simeq \odi\Delta$. If $\stY$ is inflexible then $\Pi$, $\Delta$ are affine gerbes.
\end{lem}

\begin{proof}
 The multiplication map of $f_*\odi\stY$ and its unit map lie
in $\Vect(\Gamma)\subseteq \Vect(\stX)$ and therefore determine a cover
\begin{equation}\label{eq1}
\Delta\arr \Gamma
\end{equation}
extending $f$ as claimed in the statement. Uniqueness of the extension follows from
the fact that $\Vect(\Gamma)\arr \Vect(\stX)$ is fully faithful. As $u$ is Nori
reduced we have  $u_*\odi\stX \simeq \odi\Gamma$ (see Lemma \ref{fully faithful on Vect means pushforwars O is O} or \cite[Lemma 7.11, pp. 22]{BV}). Since $\Delta\arr \Gamma$
is flat we also conclude that $v_*\odi\stY \simeq \odi\Delta$. The cover $\Pi\arr\Pi_{\sX/k}^\NN$ and its uniqueness plus
the fact $\beta_*\odi\stY\simeq \odi\Pi$ are
obtained in exactly the same way.
 
 Assume now that $\stY$ is inflexible. By definition, $\stY\arr \Delta$ factors
through a finite gerbe $\Delta'$, which can be chosen as closed substack
$\Delta'\subseteq \Delta$. But $v_*\odi\stY \simeq \odi\Delta$ which implies that $\Delta=\Delta'$ as required. Notice
that, since $\Pi^\NN_{\stX/k}\arr \Gamma$ is a quotient, the stack $\Delta$ is a gerbe
if and only if $\Pi$ is a gerbe. This completes the proof.

\end{proof}

\begin{proof}[Proof of Theorem \ref{closure of an essentially finite cover}]
    The existence of the diagram is clear: firstly $\Delta,\Gamma,u$ are by Lemma \ref{trivial section iff gerbe}, then $\Spec(k)\arr \Delta$ is the image of $y\in\sY(k)$, and finally $\sP_f$ is defined as a product. Since $u\colon \stX\arr \Gamma=\Bi G_f$ is Nori-reduced so 
is the $G_f$-torsor $\stP_f\arr\stX$. The claim about the monodromy gerbe of $\pi_*\odi{\stP_f}$ 
follows from the following fact: if $G$ is a finite group scheme over $k$, the regular 
representation $k[G]$ generates $\Rep G$ because every finite $G$-representation is a subobject 
of some $k[G]^n$.
 
The morphism $\lambda$ is faithfully flat if and only if $\Spec k\arr \Delta$ is faithfully flat. 
This is the case if and only if $\Delta$ is a gerbe. In such a situation
we have $\Delta=\Bi H$, where $H$ is a 
subgroup of $G_f$ and $\stP_f\arrdi\lambda \stY$ is an $H$-torsor. The map $u'\colon \stY\arr 
\Delta=\Bi H$ is Nori-reduced because ${u'}_*\odi\stY\simeq \odi\Delta$.

The factorization $\lambda = \eta \circ \lambda'$ arises from the factorization $\Spec k \arr 
\Delta' \arr \Delta$, where $\Delta' $ is a subgerbe of the finite stack $\Delta $.

For the last statement, $h$ induces an inclusion $\odi\stY \subset h_* \odi\stT$, and as $f$ is 
affine, an inclusion $f_*\odi\stY \subset g_* \odi\stT$. If we denote by $\langle g_* 
\odi\stT\rangle$ (respectively, $\langle\pi _* \odi{\stP_f}\rangle$) the full Tannakian sub-category of 
$\EF (\Vect (\stX))$ generated by the object $g_* \odi\stT$ (respectively, $\pi _* \odi{\stP_f}$), one 
gets the following  diagram with $2$-Cartesian square:
$$\xymatrix{
\Vect (\Bi G_f) \ar[rr] \ar[d]_{u^*} &&\Vect (\Bi G)\ar[d]^{v^*}\\
\langle\pi _* \odi{\stP_f}\rangle\ar[dr]_{x^*} \ar[rr] &&\langle g_* \odi\stT\rangle\ar[dl]^{x^*}\\
&\Vect (\Spec k)&\\
}$$
where $v \colon \stX \arr \Bi G$ corresponds to the torsor $g \colon \stT \arr \stX$, the horizontal
arrows are inclusions and the vertical arrows are equivalences. As the torsors are pointed
above $x$, the functors $x^* \circ u^*$ and $x^* \circ v^*$ are equivalent to the
forgetful functors.
Therefore, the commutativity of this last diagram proves the existence of a surjective
morphism $\varphi \colon G \arr G_f$ such that $u^*\simeq v^* \circ \varphi ^* : \Vect (\Bi G_f)
\arr \Vect (\stX)$. This implies that $u=\varphi' \circ v$, where $\varphi' \colon
\Bi G \arr \Bi G_f$ is the morphism of gerbes induced by $\varphi$.
\end{proof}

\begin{defn} Let $\sX$ be a pseudo-proper, inflexible fibered category over $k$ with a base point $x\in\sX(k)$, and let $f\colon(\sY,y)\arr(\sX,x)$ be an essentially finite cover.
    The \textit{Galois closure} of $f$ is the pointed torsor $\sP_f\arr\sX$ with the pointed map $\sP_f\arr\sY$ constructed  in Theorem \ref{closure of an essentially finite cover}.
\end{defn}

\begin{defn} \label{inflexible essentially finite covers}
Let $\sX$ be a pseudo-proper, inflexible fibered category over $k$. We call an essentially finite cover $f\colon \sY\arr \sX$ inflexible if $\sY$ is inflexible. \end{defn}

\begin{rmk}
Let $\stX$ be a pseudo-proper and inflexible algebraic stack of finite type over a
field $k$ of positive characteristic, such that $\dim_k \Hl^1(\stX,E)<\infty$ for all
vector bundles $E$ on $\stX$. Let $f\colon \stY\arr \stX$ be a Nori-reduced torsor. Then by Theorem \ref{main thm Nori gerbe of ess finite cover} (which will be proved later) $f$ is 
inflexible.
\end{rmk}

\begin{lem}   \label{GC for inflexible essentially finite covers}
    Let $\sX$ be a pseudo-proper, inflexible fibered category with a rational point $x\in\sX(k)$. Let $f\colon \sY\arr \sX$ be an inflexible essentially finite cover equipped with a rational point $y\in \sY(k)$ mapping to $x$. Then the Galois closure $\sP_f$ of $f$ satisfies:
\begin{itemize}
\item The $G_f$-torsor $\sP_f$ is Nori-reduced.
\item  For any pointed Nori-reduced torsor $g \colon (\stT,t) \arr (\stX,x)$ under a finite group scheme $G$, 
and any pointed $\stX$-morphism $h: (\stT,t) \to (\stY,y)$, there is a unique factorization $h=\lambda \circ j$, where $j\colon (\stT,t) \arr (\stP _f,p)$ is equivariant with 
respect to a surjective homomorphism $G\arr G_f$.
\end{itemize}    
 \end{lem}
 \begin{proof}
   Using  Lemma \ref{trivial section iff gerbe} one can easily show
   that when $\sY$ is inflexible any map from an essentially finite cover $\sT$ to $\sY$ is faithfully flat. In view of Theorem 
   \ref{closure of an essentially finite cover} (4)
   we can conclude the proof.
 \end{proof}

Here are some technical lemmas which will be used in proving Theorem 
\ref{main thm Nori gerbe of ess finite cover}.

\begin{lem}\label{vector bundles of an extension}
 Consider a $2$-Cartesian diagram
   \[
  \begin{tikzpicture}[xscale=1.5,yscale=-1.2]
    \node (A0_0) at (0, 0) {$\stY$};
    \node (A0_1) at (1, 0) {$\Psi$};
    \node (A1_0) at (0, 1) {$\stX$};
    \node (A1_1) at (1, 1) {$\Phi$};
    \path (A0_0) edge [->]node [auto] {$\scriptstyle{v}$} (A0_1);
    \path (A0_0) edge [->]node [auto] {$\scriptstyle{f}$} (A1_0);
    \path (A0_1) edge [->]node [auto] {$\scriptstyle{\pi}$} (A1_1);
    \path (A1_0) edge [->]node [auto] {$\scriptstyle{u}$} (A1_1);
  \end{tikzpicture}
  \]
  where $\stX, \sY,\Phi,\Psi$ are fibered categories over $k$ which admit fpqc-atlases from schemes, and $u,\pi$  are faithfully flat where $\pi$ is affine and $u$ is quasi-compact and quasi-separated such that $u_*\odi\stX\simeq \odi\Phi$. Then $v_*\odi\stY\simeq\odi\Psi$, and
the two functors $u^*\colon\Vect(\Phi)\arr\Vect(\stX)$ and
$v^*\colon\Vect(\Psi)\arr\Vect(\stY)$ are fully faithful.

A vector bundle $V\in\Vect(\stX)$ lies in the essential image of 
$u^*\colon\Vect(\Phi)\arr\Vect(\stX)$ if and only if $f^*V$ comes from a
vector bundle on $\Psi$.

If $\pi$ is a surjective cover, a vector bundle $V\in\Vect(\stY)$ lies  in the  essential image of  $v^*\colon\Vect(\Psi)\arr\Vect(\stY)$ if and only if $f_*V$ comes from a vector bundle on $\Phi$.
\end{lem}

\begin{proof}
By \cite[Lemma 7.17]{BV} and flat base change it follows that 
\begin{itemize}
\item $v_*\odi\stY\simeq\odi\Psi$,

\item $u^*\colon\Vect(\Phi)\arr\Vect(\stX)$ and 
$v^*\colon\Vect(\Psi)\arr\Vect(\stY)$ are fully faithful.
\end{itemize}
Denote by $\shD$ and $\shC$ the essential images of these $u^*$ and $v^*$ respectively.
 
Let $V\in\Vect(\stX)$. We must show that $V\in \shD$ if and only if $f^*V\in 
\shC$. The ``only if'' part is clear. Conversely, suppose that $f^* V= v^* W$ with
$W \,\in \, \Vect(\Psi )$, and consider the
canonical homomorphism $u^* u_* V \arr V$; pulling back by $f$ one 
gets $$v^*v_*(f^*V) =v^*v_*(v^*W) \arr v^*W=f^*V\, .$$ This
homomorphism is an isomorphism because
$v_* \odi\stY \simeq \stO _\Psi$. As $f$ is faithfully flat, one concludes that $V 
\simeq u^*u_* V$, and as $u$ is faithfully flat, it follows that $u_*V$ is a 
vector bundle. Thus we have $V\in\shD$.

Assume now that $\pi$ is a surjective cover;
consequently $f$ is also a surjective cover. In particular, $\pi_*$ and $f_*$
send vector bundles to vector bundles. Given $V\in\Vect(\stY)$ we must show that $V\in \shC$ if and only if $f_*V\in\shD$. 
The ``only if'' part is easy: if $W\in \Vect(\Psi)$ then $f_*(v^*W)\simeq u^*\pi_*W$
because $\pi$ is affine.

For the converse, assume that $f_*V\in \shD$, meaning 
$f_*V$ comes from a vector bundle on $\Phi$. Since $u_*\odi\stX\simeq 
\odi\Phi$ it follows that $u_*(f_*V)$ is a vector bundle and the canonical
homomorphism $u^*u_* 
(f_*V) \arr (f_*V)$ is an isomorphism. This homomorphism can also be obtained
by applying 
$f_*$ to the canonical homomorphism $v^*v_*V\arr V$. Since $f$ is affine this means that 
the previous homomorphism is an isomorphism. To conclude that $V\in \shC$ it
suffices to 
show that $v_*V$ is a vector bundle. But $v$ is faithfully flat and $v^*(v_*V)$ is
a vector bundle. Now by descent it follows that $v_*V$ is also a vector bundle.
\end{proof}

\begin{rmk}\label{pull back by a torsor}
Consider a $G$-torsor $f \colon \stY \to \stX$ for an affine group scheme $G$, where
$\stX$ is a quasi-compact and quasi-separated algebraic stack, and the
corresponding $2$-Cartesian diagram
  \[
  \begin{tikzpicture}[xscale=1.5,yscale=-1.2]
    \node (A0_0) at (0, 0) {$\stY$};
    \node (A0_1) at (1, 0) {$\Spec k$};
    \node (A1_0) at (0, 1) {$\stX$};
    \node (A1_1) at (1, 1) {$\Bi G$};
    \path (A0_0) edge [->]node [auto] {$\scriptstyle{v}$} (A0_1);
    \path (A0_0) edge [->]node [auto] {$\scriptstyle{f}$} (A1_0);
    \path (A0_1) edge [->]node [auto] {$\scriptstyle{}$} (A1_1);
    \path (A1_0) edge [->]node [auto] {$\scriptstyle{u}$} (A1_1);
  \end{tikzpicture}
  \]
We see that $\Hl^0(\odi\stY)=k$ if and only if $u_*\odi\stX \simeq \odi{\Bi G}$. In
this case, applying Lemma \ref{vector bundles of an extension}, we conclude that $u^*\colon \Vect(\Bi G)\arr \Vect(\stX)$ is fully faithful with essential image the category of vector bundles $V$ such that $f^*V$ is trivial.
\end{rmk}

\begin{lem}[{\cite[p.~264, Lemma 1]{Nori2}}]\label{morphisms of torsors}
Let $\stX$ be a quasi-compact and quasi-separated algebraic algebraic stack and
  \[
  \begin{tikzpicture}[xscale=0.9,yscale=-1.0]
    \node (A0_0) at (0, 0) {$\stZ$};
    \node (A0_2) at (2, 0) {$\stY$};
    \node (A1_1) at (1, 1) {$\stX$};
    \path (A0_0) edge [->]node [auto] {$\scriptstyle{h}$} (A0_2);
    \path (A0_2) edge [->]node [auto] {$\scriptstyle{g}$} (A1_1);
    \path (A0_0) edge [->]node [auto,swap] {$\scriptstyle{f}$} (A1_1);
  \end{tikzpicture}
  \]
a $2$-commutative diagram, where $f$ and $g$ are torsors for affine group schemes $G$ and $H$
respectively.  Suppose that $\Hl^0(\odi\stZ)=k$. Then there exists a homomorphism $\varphi \colon
G \arr H$ inducing $h$.

Moreover, $h$ is faithfully flat if and only if $\varphi \colon G \arr H$ is faithfully flat, in
which case $h\colon \stZ \arr \stY$ is a torsor for the kernel of $\varphi$. If $H/\Imm(\varphi)$ is
affine (e.g. if $H$ or $G$ are finite) then this is also equivalent to the statement
that $\Hl^0(\odi\stY)=k$. 
\end{lem}

\begin{proof}
Consider the morphisms $u \colon \stX \arr \Bi G$ and $v\colon \stX \arr \Bi H$ corresponding to the
torsors $f$ and $g$ respectively. Since $\Hl^0(\sO_\sZ)=k$ we have  $u_*\sO_\sX\simeq
\sO_{\Bi G}$, and hence by \ref{fully faithful on Vect means pushforwars O is O} the pullback functor $u^*\colon \Vect(\Bi G)\arr \Vect(\sX)$ is fully faithful. The
objects of the essential image of $v^* \colon \Vect (\Bi H) \arr \Vect (\stX )$ are trivialized by
$g$ and thus by $f$. From Remark \ref{pull back by a torsor} we obtain a factorization
\[
v^*\colon \Vect(\Bi H)\arr \Vect(\Bi G) \subseteq  \Vect(\stX)
\]
which, by Tannakian duality, is induced by a factorization $v\colon \stX\arrdi u \Bi G \arrdi \gamma \Bi H$. Consider the $2$-Cartesian diagrams
  \[
  \begin{tikzpicture}[xscale=1.5,yscale=-1.2]
    \node (A0_0) at (0, 0) {$\stZ$};
    \node (A0_1) at (1, 0) {$\Spec k$};
    \node (A1_0) at (0, 1) {$\stY$};
    \node (A1_1) at (1, 1) {$U$};
    \node (A1_2) at (2, 1) {$\Spec k$};
    \node (A2_0) at (0, 2) {$\stX$};
    \node (A2_1) at (1, 2) {$\Bi G$};
    \node (A2_2) at (2, 2) {$\Bi H$};
    \path (A0_1) edge [->,dashed]node [auto] {$\scriptstyle{w}$} (A1_1);
    \path (A0_0) edge [->]node [auto] {$\scriptstyle{}$} (A0_1);
    \path (A2_0) edge [->]node [auto] {$\scriptstyle{u}$} (A2_1);
    \path (A1_0) edge [->]node [auto] {$\scriptstyle{a}$} (A1_1);
    \path (A1_1) edge [->]node [auto] {$\scriptstyle{}$} (A1_2);
    \path (A1_0) edge [->]node [auto] {$\scriptstyle{g}$} (A2_0);
    \path (A1_1) edge [->]node [auto] {$\scriptstyle{g'}$} (A2_1);
    \path (A0_0) edge [->]node [auto] {$\scriptstyle{h}$} (A1_0);
    \path (A2_1) edge [->]node [auto] {$\scriptstyle{\gamma}$} (A2_2);
    \path (A1_2) edge [->]node [auto] {$\scriptstyle{}$} (A2_2);
  \end{tikzpicture}
  \]
We claim that there exists a dashed arrow $w$ as above making the upper diagram $2$-Cartesian. This
would imply that the functor $\gamma\colon \Bi G\arr \Bi H$ is induced by a group homomorphism  $\varphi \colon G \arr H$.

Set $f'\colon \Spec k\arr \Bi G$. Consider the map of $\sO_\sX$-algebras $\lambda\colon 
g_*\sO_\sY\arr f_*\sO_\sZ$. Applying $u_*$ to $\lambda$ and using $a_*\sO_\sY\simeq \sO_U$, 
$\Hl^0(\sO_\sZ)=k$, we get a map $${g'}_*\sO_U\cong {g'}_* a_*\sO_\sY\cong u_* g_*\sO_\sY\arr 
u_*f_*\sO_\sZ\cong {f'}_*\sO_{\Spec(k)}\, .$$ Applying $\Spec_{\Bi G}(-)$ on both sides we get the 
arrow $w$.

To prove that $h$ is the pullback of $w$ we just have to show that the adjunction maps 
$$u^*u_*g_*\sO_\sY\arr g_*\sO_\sY\hspace{20pt}\text{and}\hspace{20pt} u^*u_*{f}_*\sO_\sZ\arr 
f_*\sO_\sZ$$ are isomorphisms. But the adjunction map for $g_*\sO_\sY$ coincides with the following 
composition: $$u^*u_*g_*\sO_\sY\cong u^*{g'}_*a_*\sO_\sY\cong u^*{g'}_*\sO_U\cong g_*\sO_\sY$$ 
which is an isomorphism, and the same is true for ${f}_*\sO_\sZ$.

The map $h$ is faithfully flat if and only if $w$ is faithfully flat. By
Lemma \ref{key product gerbes} this is the case if and only if $\varphi\colon G\arr H$ is
surjective, so that $U=\Bi (\Ker(\varphi))$. The condition that $\Hl^0(\odi\stY)=k$ is equivalent to
the condition that $v_*\odi\stX \simeq \odi{\Bi H}$ and, by Lemma
\ref{fully faithful on Vect means pushforwars O is O}, to the full faithfulness of the functor
$\gamma^*\colon \Vect(\Bi H)\arr \Vect(\Bi G)$. This last condition is equivalent to the surjectivity of $\varphi\colon G\arr H$ when $H/\Imm(\varphi)$ is affine (see \cite[Remark B.7]{TZ}).
\end{proof}

\begin{lem}\label{decomposition series lemma}
Let $\stX$ be an algebraic stack over a field $k$, $A$ a finite and local $k$-algebra with residue
field $k$ and $\shF\in \Vect(\stX\times \Spec A)$. Let $\stX\arrdi i \stX\times \Spec A \arrdi f
\stX$ be the maps corresponding to $k\arr A\arr k$. Then there is a sequence of surjective maps of
vector bundles
 \[
 f_*\shF=\shG_N \arr \shG_{N-1} \arr \cdots \arr \shG_1 \arr \shG_0=0
\]
such that $\Ker(\shG_l \arr \shG_{l-1})\simeq i^*\shF$ for all $1\leq l \leq N$.
\end{lem}
\begin{proof}
 Consider a decomposition series of $A$-modules
\[
A=A_N \arr A_{N-1} \arr \cdots \arr A_1 \arr A_0=0\, ;
\]
the above maps are surjective with $\Ker(A_l \arr A_{l-1})\simeq k$ as $A$-modules for all $1\leq l \leq N$. Let $p\colon \stX\times \Spec A\arr \Spec A$ be the projection; consider the functor 
$$
\Psi=f_*(\shF\otimes p^*(-)) \colon \Mod A \arr \QCoh(\stX)\, .
$$
Since $p$ is flat, $\shF$ is a vector bundle, and as $f$ is affine the functor $\Psi$ is exact. Moreover $\Psi(A)=f_*\shF$ and
\[
\Psi(k)=f_*(\shF\otimes p^*k) \simeq f_*(\shF\otimes i_*\odi\stX)\simeq f_*i_*i^*\shF \simeq i^*\shF\, .
\]
Applying $\Psi$ to the above sequence of $A$-modules we find the desired sequence.
\end{proof}

\begin{lem}\label{push of ess finite is still ess finite}
 Let $\stX$ be a pseudo-proper and inflexible algebraic stack of finite type over a
field $k$ of positive characteristic, such that $\dim_k \Hl^1(\stX,E)<\infty$ for all
vector bundles $E$ on $\stX$. Let $f\colon \stY\arr \stX$ be an essentially finite
cover. Then for all maps $\phi\colon \stY\arr \Phi$ to a finite stack over $k$, and
for all $W\in \Vect(\Phi)$, the vector bundle $f_*\phi^*W$ is essentially finite
in $\Vect(\stX)$.
\end{lem}

\begin{proof}
To avoid problems with different ranks, we first observe that the rank of $W$
can be assumed to be constant, for instance, by considering the connected components of
$\Phi$. By Lemma \ref{trivial section iff gerbe} we have a Cartesian diagram
  \[
  \begin{tikzpicture}[xscale=1.8,yscale=-1.2]
    \node (A0_0) at (0, 0) {$\stY$};
    \node (A0_1) at (1, 0) {$\Delta$};
    \node (A1_0) at (0, 1) {$\stX$};
    \node (A1_1) at (1, 1) {$\Gamma$};
    \path (A0_0) edge [->]node [auto] {$\scriptstyle{v}$} (A0_1);
    \path (A1_0) edge [->]node [auto] {$\scriptstyle{u}$} (A1_1);
    \path (A0_1) edge [->]node [auto] {$\scriptstyle{}$} (A1_1);
    \path (A0_0) edge [->]node [auto] {$\scriptstyle{f}$} (A1_0);
  \end{tikzpicture}
  \]
where $\Gamma$ is a finite gerbe.
We will prove that there exists a surjective cover $s\colon \stX'\arr \stX$ such
that $s^*f_*\phi^*W$ is free on the connected components of $\stX'$;
this will be done by only assuming that there is a Cartesian diagram as above
with $\Gamma$ a finite gerbe and $f$ finite, but without requiring that $\stX$
is inflexible, as this will allow us to replace $\sX$ by any surjective cover of it. The
lemma then follows from Theorem \ref{key thm}.

As mentioned above, if $s'\colon \stT'\arr \stX$ is a surjective cover 
we can always replace $\stX$ by $\stT'$. Let $L/k$ be a finite field extension 
with a map $\Spec L\arr \stY$. The base change of $\Spec L \arr \stY \arr \Delta 
\arr \Gamma$ along $\stX\arr \Gamma$ is a surjective cover $\stQ\arr \stX$. 
Replacing $\stX$ by $\stQ$ we can assume that $\stX\arr \Gamma$ factors as 
$\stX\arr \Spec L \arr \Gamma$. This means that $\stY\arrdi f \stX$ is the 
projection $\stX\times_L A\arr \stX$, where $A/L$ is a finite $L$-algebra. Since 
$\phi$ factors as $\stY \arr \Phi\times_k L \arr \Phi$ we can assume that $L=k$. 
Extending again $k$ we can further assume that $A$ is a product of local 
$k$-algebras with residue field $k$. Splitting $\stY$ according to the 
decomposition of $A$ we can assume that $A$ is also local. Let $i\colon \stX \arr 
\stY=\stX\times A$ be the inclusion corresponding to $A\arr k$.

Finally consider a surjective cover $T\arr \Phi$ from a finite scheme and the $2$-Cartesian diagram
  \[
  \begin{tikzpicture}[xscale=1.8,yscale=-1.2]
    \node (A0_0) at (0, 0) {$\stZ$};
    \node (A0_2) at (2, 0) {$T$};
    \node (A1_0) at (0, 1) {$\stX$};
    \node (A1_1) at (1, 1) {$\stY$};
    \node (A1_2) at (2, 1) {$\Phi$};
    \path (A0_2) edge [->]node [auto] {$\scriptstyle{}$} (A1_2);
    \path (A1_0) edge [->]node [auto] {$\scriptstyle{i}$} (A1_1);
    \path (A0_0) edge [->]node [auto] {$\scriptstyle{}$} (A1_0);
    \path (A1_1) edge [->]node [auto] {$\scriptstyle{\phi}$} (A1_2);
    \path (A0_0) edge [->]node [auto] {$\scriptstyle{}$} (A0_2);
  \end{tikzpicture}
  \]
Replacing $\stX$ by $\stZ$ we can assume that $\phi \circ i\colon \stX\arr \Phi$
factors through a finite scheme, which in particular implies that $i^*\phi^*W$ is free.
Applying Lemma \ref{decomposition series lemma} to $\shF=\phi^*W$ we obtain a sequence
of surjective homomorphisms of quasi-coherent sheaves on $\stX$
\[
f_*(\phi^*W)=\shG_N \arr \shG_{N-1} \arr \cdots \arr \shG_1 \arr \shG_0=0
\]
with free kernels. The final cover trivializing $f_*(\phi^*W)$ exists
due to Lemma \ref{extensions of O}.
\end{proof}

\begin{rmk}
Assume $\car k = 0$. Then the proof of Lemma \ref{push of ess finite is still ess finite}
would not work anymore. The problem is that  Lemma \ref{extensions of O}  holds only in positive characteristic:

Let $\sE$xt$(\sO,\sO)$ be the stack over $\Aff/k$ which associates with each $T\in\Aff/k$ the groupoid of extensions of the form 
$$0\arr \sO_T\arr E\arr \sO_T\arr 0$$ where $E$ is a quasi-coherent sheaf on $T$. A morphism between two extensions is given by a commutative diagram $$\xymatrix{0\ar[r]&\sO_T\ar[r]\ar@{=}[d]&E\ar[d]^\phi\ar[r]&\sO_T\ar@{=}[d]\ar[r]&0\\0\ar[r]&\sO_T\ar[r]&E'\ar[r]&\sO_T\ar[r]&0}$$
One checks readily that $\sE$xt$(\sO,\sO)$ is a trivial gerbe under $\G_a$, or in other words $\sE$xt$(\sO,\sO)=\Bi \G_a$.
Now let $\stX=X$ be a proper geometrically connected reduced scheme over $k$ which admits a non-trivial extension of quasi-coherent sheaves:
$$0\arr \sO_X\arr E\arr \sO_X\arr 0$$
If the conclusion of Lemma \ref{extensions of O} was true in this case, then $E$ is essentially finite by Theorem \ref{key thm}. Let $\Gamma$ be the monodromy gerbe of $E$ in $\Ess(\Vect(X))$. Then we have the following commutative diagram 
$$\xymatrix{X\ar[rr]\ar[dr]&&\sE\text{xt}(\sO,\sO)=\Bi\G_a\\&\Gamma\ar[ur]_-\lambda&}$$
due to the full faithfulness of the pullback $\Vect(\Gamma)\arr \Ess(\Vect(X))$ and the fact that $\sO_X$ and $E$ come from  $\Vect(\Gamma)$. The map $\lambda$ is non-trivial because its composition with $X\arr \Gamma$ is non-trivial. But by \cite[Chapitre IV, \S 2, $n^o$1, 1.1,  pp. 483]{DG}  $\G_a$ has no non-trivial finite subgroup in characteristic 0 and therefore $\lambda$ has to be trivial, a contradiction.

However, Lemma \ref{push of ess finite is still ess finite} still holds true in characteristic 0 when $\sY$ is inflexible. This is due to Lemma \ref{push of ess finite for etale} and Proposition \ref{main thm ess finite etale} below.
\end{rmk}

\begin{lem}\label{push of ess finite for etale}
 Let $\stX$ be a pseudo-proper and inflexible algebraic stack of finite type over $k$,
and let $f\colon \stY\arr \stX$ be an \'etale surjective cover. Then for all maps
$\phi\colon \stY\arr \Phi$ to a finite (respectively, finite and \'etale) stack over
$k$ and for all $W\in \Vect(\Phi)$, the vector bundle $f_*(\phi^*W)$ is essentially
finite (respectively, essentially finite with \'etale monodromy gerbe) in
$\Vect(\sX)$. In particular, $f$ is essentially finite.
\end{lem}

\begin{proof}
 There exists a Cartesian diagram
   \[
  \begin{tikzpicture}[xscale=1.8,yscale=-1.2]
    \node (A0_0) at (0, 0) {$\sqcup_i \stZ$};
    \node (A0_1) at (1, 0) {$\stY$};
    \node (A1_0) at (0, 1) {$\stZ$};
    \node (A1_1) at (1, 1) {$\stX$};
    \path (A0_0) edge [->]node [auto] {$\scriptstyle{a}$} (A0_1);
    \path (A1_0) edge [->]node [auto] {$\scriptstyle{r}$} (A1_1);
    \path (A0_1) edge [->]node [auto] {$\scriptstyle{f}$} (A1_1);
    \path (A0_0) edge [->]node [auto] {$\scriptstyle{b}$} (A1_0);
  \end{tikzpicture}
  \]
where $r\colon \stZ\arr \stX$ is an \'etale surjective cover. Since $a^*\phi^*W$ also
comes from the finite stack $\Phi$, taking a finite atlas of $\Phi$ we can find a
surjective cover $\lambda\colon \sU\to \bigsqcup_i \stZ$ which trivializes it. Denote
by $\sU_i$ the inverse image of the $i$-th piece of $\bigsqcup_i \stZ$ under $\lambda$. Then $r^*f_*\phi^*W=b_*a^*\phi^*W$   is trivialized by the surjective cover $\sU_1\times_\sZ\times\cdots\times_\sZ\sU_n\arr \sZ$. Thus, by Remark
\ref{key thm for torsors}, $f_*\phi^*W$ is essentially finite and it has an \'etale
monodromy gerbe if $\Phi$ is \'etale (so that $\lambda$ can also be chosen
to be \'etale).
\end{proof}

\begin{proof}[Proof of Theorem \ref{main thm Nori gerbe of ess finite cover}]
 By Lemma \ref{trivial section iff gerbe} there are  $2$-Cartesian diagrams
  \[
  \begin{tikzpicture}[xscale=1.8,yscale=-1.2]
    \node (A0_0) at (0, 0) {$\stY$};
    \node (A0_1) at (1, 0) {$\Pi$};
    \node (A0_2) at (2, 0) {$\Delta$};
    \node (A1_0) at (0, 1) {$\stX$};
    \node (A1_1) at (1, 1) {$\Pi^\NN_\stX$};
    \node (A1_2) at (2, 1) {$\Gamma$};
    \path (A0_0) edge [->]node [auto] {$\scriptstyle{\beta}$} (A0_1);
    \path (A0_1) edge [->]node [auto] {$\scriptstyle{}$} (A0_2);
    \path (A1_0) edge [->]node [auto] {$\scriptstyle{\alpha}$} (A1_1);
    \path (A1_1) edge [->]node [auto] {$\scriptstyle{}$} (A1_2);
    \path (A0_2) edge [->]node [auto] {$\scriptstyle{}$} (A1_2);
    \path (A0_0) edge [->]node [auto] {$\scriptstyle{f}$} (A1_0);
    \path (A0_1) edge [->]node [auto] {$\scriptstyle{\pi}$} (A1_1);
  \end{tikzpicture}
  \]
where $\Gamma$ is the monodromy gerbe of $f_*\odi\stY\in \EF(\Vect(\stX))$. It also follows from Lemma \ref{trivial section iff gerbe} that $\Pi$ is  a gerbe
if $\stY$ is inflexible.

For the converse assume that $\Pi$ is a gerbe. By Remark \ref{push of O Nori gerbe} and
Lemma \ref{vector bundles of an extension}, the pullback functor $\beta^*\colon \Vect(\Pi)\arr \Vect(\stY)$ is fully faithful with essential image the full subcategory $\shC$ of $\Vect(\stY)$ of vector bundles $V$ such that $f_*V\in\EF(\Vect(\stX))$.
Since $\Pi\arr \Pi^\NN_\stX$ is faithful the gerbe $\Pi$ is profinite, so
we have $\shC\subseteq \EF(\Vect(\stY))$.
We will show that this is an equality and that $\stY$ is inflexible. This
would immediately imply that $\stY\arr \Pi$ is the Nori fundamental gerbe. 
Let $\phi\colon \stY\arr \Phi$ be a map to a finite stack. Using
Lemma \ref{push of ess finite for etale} (notice that since $\Pi$ is gerbe, $\Delta$ is also a gerbe, so if $\car k=0$ then
$\Delta\arr\Gamma$ and $f\colon\sY\arr\sX$ are \'etale covers) and Lemma \ref{push of ess finite is still ess finite} it
follows that $\phi^*W \in \shC$ for all $W\in \Vect(\Phi)$. Thus the pullback
by $\phi\colon \stY \arr \Phi$ has a factorization
\[
\Vect(\Phi) \arr \Vect(\Pi) \simeq \shC \subseteq \Vect(\stY)\, .
\]
By Remark \ref{Tannakian reconstruction} one gets a factorization $\stY\arr \Pi \arr \Phi$ as
required. This shows that $\stY\arr \Pi$ is a Nori fundamental gerbe and, in particular, $\stY$
is inflexible. Since $\stY$ is pseudo-proper, all essentially finite vector bundles on $\stY$ are
pullbacks from some finite gerbe. Thus the above factorization also implies the equality
$\shC\,=\,\EF(\Vect(\stY))$.

Notice that if $\stY$ is inflexible then one always has $\Hl^0(\odi\stY)=k$ (see
Remark \ref{infl then HO=k}).

\emph{The \'etale case.} Assume that $f$ is \'etale and that $\Hl^0(\odi\stY)=k$. By
Lemma \ref{vector bundles of an extension} we have $\Hl^0(\odi\Delta)=k$. In particular,
$\Delta$ is geometrically connected. Since $\Delta\arr \Gamma$ is \'etale, it follows that $\Delta$ is
also geometrically reduced. Using \cite[Proposition 4.3]{BV} we conclude that $\Delta$ and
therefore $\Pi$ are gerbes. Thus $\stY$ is inflexible. By Lemma
\ref{push of ess finite for etale} we see that $\Gamma$ is \'etale. In particular
there are $2$-Cartesian diagrams
  \[
  \begin{tikzpicture}[xscale=2,yscale=-1.2]
    \node (A0_0) at (0, 0) {$\stY$};
    \node (A0_1) at (1, 0) {$\Pi^\NN_\stY$};
    \node (A0_2) at (2, 0) {$\Pi'$};
    \node (A0_3) at (3, 0) {$\Delta$};
    \node (A1_0) at (0, 1) {$\stX$};
    \node (A1_1) at (1, 1) {$\Pi^\NN_\stX$};
    \node (A1_2) at (2, 1) {$\Pi^{\NN,\et}_\stX$};
    \node (A1_3) at (3, 1) {$\Gamma$};
    \path (A0_0) edge [->]node [auto] {$\scriptstyle{}$} (A0_1);
    \path (A0_1) edge [->]node [auto] {$\scriptstyle{}$} (A0_2);
    \path (A1_0) edge [->]node [auto] {$\scriptstyle{}$} (A1_1);
    \path (A0_3) edge [->]node [auto] {$\scriptstyle{}$} (A1_3);
    \path (A1_1) edge [->]node [auto] {$\scriptstyle{}$} (A1_2);
    \path (A0_2) edge [->]node [auto] {$\scriptstyle{}$} (A1_2);
    \path (A0_0) edge [->]node [auto] {$\scriptstyle{f}$} (A1_0);
    \path (A0_1) edge [->]node [auto] {$\scriptstyle{}$} (A1_1);
    \path (A1_2) edge [->]node [auto] {$\scriptstyle{}$} (A1_3);
    \path (A0_2) edge [->]node [auto] {$\scriptstyle{}$} (A0_3);
  \end{tikzpicture}
  \]
We want to show that $\Pi'\,=\,\Pi^{\NN,\et}_\stY$. By Lemma \ref{vector bundles of an extension} the pullback $\Vect(\Pi')\arr \Vect(\stY)$ is fully faithful and its essential image $\shC$ consists of vector bundles $V\in\Vect(\stY)$ such that $f_*V$ comes from $\Pi^{\NN,\et}_\stX$. One must show that the essential image $\shD$ of the fully faithful map $\Vect(\Pi^{\NN,\et}_\stY)\arr \Vect(\stY)$ coincides with $\shC$. The vector bundles in $\shD$ are the essentially finite vector bundles with \'etale monodromy gerbes. Since the map $\Pi'\arr\Pi^{\NN,\et}_\stX$ is faithful, it follows that $\Pi'$ is pro\'etale, so that $\shC\subseteq \shD$. The opposite inclusion instead follows
from Lemma \ref{push of ess finite for etale}. 

\emph{The torsor case.} Assume that $f$ is a torsor for a finite group scheme $G$. Since the regular
representation $k[G]$ generates $\Rep(G)$ as $k$-Tannakian category, there are $2$-Cartesian diagrams
$$
\xymatrix{\sY\ar[rr]^-\beta\ar[d]^-f&&\Pi\ar[d]^-\pi\ar[rr]&&\Delta\ar[rr]\ar[d]&&\Spec(k)\ar[d]\\ \sX\ar[rr]^-\alpha && \Pi_{\sX}^\NN\ar@{>>}[rr]&&\Gamma\ar[rr]^-{}&&\Bi G}
$$
and the map $\Gamma\arr \Bi G$ is faithful and, by \cite[Remark B.7]{TZ}, affine. In particular
$\Delta$ is a finite scheme, and by Lemma \ref{vector bundles of an extension} we can conclude
that $\Delta=\Spec(\Hl^0(\odi\stY))$. Now the theorem follows because $\stY$ is inflexible if and only if $\Delta$ is a gerbe.
 \end{proof}

\begin{proof}[Proof of Corollary \ref{main thm Nori gerbe pointed case}]
 The Nori gerbe of $\stX$ is $\Pi^\NN_\stX=\Bi \pi^\NN(\stX,x)$, and $\stX\arr \Bi \pi^\NN(\stX,x)$
maps $x$ to the trivial torsor. If $(\stY,y)\arrdi f (\stX,x)$ is an essentially finite cover, then
the extension $\Pi\arr \Pi^\NN_\stX$ defined in Theorem
\ref{main thm Nori gerbe of ess finite cover} is described by the $2$-Cartesian diagrams
   \[
  \begin{tikzpicture}[xscale=2.6,yscale=-1.2]
    \node (A0_0) at (0, 0) {$\Spec k$};
    \node (A0_1) at (1, 0) {$\pi^\NN(\stX,x)$};
    \node (A0_2) at (2, 0) {$\widetilde\stX$};
    \node (A0_3) at (3, 0) {$\Spec k$};
    \node (A1_1) at (1, 1) {$\stY_x$};
    \node (A1_2) at (2, 1) {$\stY$};
    \node (A1_3) at (3, 1) {$[\stY_x/\pi^\NN(\stX,x)]$};
    \node (A2_1) at (1, 2) {$\Spec k$};
    \node (A2_2) at (2, 2) {$\stX$};
    \node (A2_3) at (3, 2) {$\Bi\pi^\NN(\stX,x)$};
    \path (A0_0) edge [->]node [auto] {$\scriptstyle{1}$} (A0_1);
    \path (A0_1) edge [->]node [auto] {$\scriptstyle{}$} (A1_1);
    \path (A0_0) edge [->]node [auto] {$\scriptstyle{y}$} (A1_1);
    \path (A1_3) edge [->]node [auto] {$\scriptstyle{}$} (A2_3);
    \path (A0_1) edge [->]node [auto] {$\scriptstyle{}$} (A0_2);
    \path (A1_1) edge [->]node [auto] {$\scriptstyle{}$} (A1_2);
    \path (A2_2) edge [->]node [auto] {$\scriptstyle{}$} (A2_3);
    \path (A0_3) edge [->]node [auto] {$\scriptstyle{}$} (A1_3);
    \path (A0_2) edge [->]node [auto] {$\scriptstyle{}$} (A1_2);
    \path (A1_1) edge [->]node [auto] {$\scriptstyle{}$} (A2_1);
    \path (A1_2) edge [->]node [auto] {$\scriptstyle{}$} (A1_3);
    \path (A2_1) edge [->]node [auto] {$\scriptstyle{x}$} (A2_2);
    \path (A0_2) edge [->]node [auto] {$\scriptstyle{}$} (A0_3);
    \path (A1_2) edge [->]node [auto] {$\scriptstyle{f}$} (A2_2);
  \end{tikzpicture}
  \]
that is $\Pi=[\stY_x/\pi^\NN(\stX,x)]$. Again by Theorem \ref{main thm Nori gerbe of ess finite cover} we know that $\stY$ is inflexible if and only if $\Pi$ is a gerbe. On the other hand, the
following three conditions are equivalent:
\begin{itemize}
\item $\Pi$ is a gerbe,

\item $\Spec k \arr \Pi$ is faithfully flat, and

\item the orbit map $\pi^\NN(\stX,x)\arr \stY_x$ of $y$ is faithfully flat.
\end{itemize}
When these equivalent conditions hold, by Theorem \ref{main thm Nori gerbe of ess finite cover} we
have $\Pi=\Bi\pi^\NN(\stY,y)$ and $\stY_x \simeq \pi^\NN(\stX,x)/\pi^\NN(\stY,y)$. The equivalence of categories in the statement follows
easily from Proposition \ref{equiv for essentially finite covers}.

Let $f\colon (\stY,y)\arr (\stX,x)$ be an essentially finite cover with $\stY$ inflexible. If $f$ is
a torsor under a group $G$, using Lemma \ref{key product gerbes} (1) the last diagram in Theorem
\ref{main thm Nori gerbe of ess finite cover} tells that $\pi^\NN(\stY,y)$ is normal in
$\pi^\NN(\stX,x)$ with quotient $G$.
Conversely, if $\pi^\NN(\stY,y)$ is normal in $\pi^\NN(\stX,x)$ with quotient $G$, then using again
Lemma \ref{key product gerbes} (1), $\Bi \pi^\NN(\stY,y) \arr \Bi \pi^\NN(\stX,x)$ and its base change
$f\colon \stY\arr \stX$ is torsor under $G$.

For the last claim, considering the above Cartesian diagrams we see that the following three are
equivalent:
\begin{itemize}
\item $f$ is \'etale,

\item $\Pi=\Bi \pi^\NN(\stY,y)\arr \Bi\pi^\NN(\stX,x)$ is \'etale,

\item $\stY_x\simeq \pi^\NN(\stX,x)/\pi^\NN(\stY,y)\arr \Spec k$ is \'etale.
\end{itemize}
This completes the proof. 
\end{proof}

\begin{prop}\label{main thm ess finite etale}
 Let $\stX$ be a pseudo-proper and inflexible fibered category over $k$ and $f\colon\stY\arr\stX$ be a cover. The following are equivalent:
 \begin{enumerate}
  \item $\stY$ is inflexible, $f$ is essentially finite and $f_*\odi\stY$ has \'etale monodromy gerbe in $\EF(\Vect(\stX))$.
  \item $f$ is \'etale and $\Hl^0(\odi\stY)=k$.
 \end{enumerate}
\end{prop}

\begin{proof}
The implication $(2)\then (1)$ follows from Lemma \ref{push of ess finite for etale} and Theorem \ref{main thm Nori gerbe of ess finite cover}.
For the converse, let $\Delta\arr \Gamma$ be as in Lemma \ref{trivial section iff gerbe}. Since $\stY$ is inflexible it follows that $\Delta$ is a gerbe. Moreover by hypothesis $\Gamma$ is \'etale. Thus $\Delta$ is \'etale too because $\Delta\arr\Gamma$ is faithful. By base change it follows that $f$ is \'etale.
\end{proof}

\section{Towers of torsors and their Galois closures}

Let $k$ be a base field and $G$ and $H$ be finite group schemes over $k$. In this section we 
introduce the notion of tower of torsors and Galois closure of a tower of torsors. At the end of 
the section Theorems \ref{Galois closure when G or H is etale} and \ref{main theorem 
for Galois closure} will be proved.

\begin{defn}\label{the stack B(G,H)}
A $(G,H)$-tower of torsors over a fibered category $\stX$ over $k$ is a sequence 
of map of fibered categories $\stZ\arrdi h \stY\arrdi g \stX$ where $g$ is a 
$G$-torsor and $h$ is an $H$-torsor. When $G,H$ are clear from the context we will 
just talk about a tower of torsors. The $(G,H)$-tower is called pointed over $k$ if
$\stZ\arr \stX$ (and therefore $\stY\arr \stX$) is a pointed cover over $k$.
 
We define the stack $\Bi(G,H)$ as the stack over $\Aff/k$ whose section
over an affine scheme $U$ is the groupoid of $(G,H)$-tower of torsors
over $U$. A morphism between two $(G,H)$-towers $\stZ\arrdi h \stY\arrdi g \stX$
and $\sZ\arrdi {h'}\sY'\arrdi {g'}\sX$ is  a commutative diagram
$$\xymatrix{\sZ\ar[r]^b\ar[d]_-h&\sZ'\ar[d]^-{h'}\\\sY\ar[r]^a&\sY'}$$
of $\sX-$maps such that $a$ is $G$-equivariant and $\sZ\arr \sZ'\times_{\sY'}\sY$ is $H$-equivariant. Clearly, $a,b$ are isomorphisms so that
the stack $\Bi(G,H)$ is a groupoid.

\end{defn}

Let us start with a preliminary remark:

\begin{rmk}\label{universal tower of BGH}
 A tower of torsors over a fibered category $\stX$ is the same as a map $\stX\arr \Bi(G,H)$. Moreover $\Bi(G,H)$ has a universal tower given by
   \[
  \begin{tikzpicture}[xscale=2.7,yscale=-1.2]
    \node (A0_1) at (1, 0) {$\stW$};
    \node (A0_2) at (2, 0) {$\Spec k$};
    \node (A1_1) at (1, 1) {$\Homsh_k(G,\Bi H)$};
    \node (A1_2) at (2, 1) {$\Bi H$};
    \node (A1_3) at (3, 1) {$\Spec k$};
    \node (A2_1) at (1, 2) {$\Bi(G,H)$};
    \node (A2_3) at (3, 2) {$\Bi G$};
    \path (A1_3) edge [->]node [auto] {$\scriptstyle{}$} (A2_3);
    \path (A0_1) edge [->]node [auto] {$\scriptstyle{}$} (A0_2);
    \path (A1_1) edge [->]node [auto] {$\scriptstyle{u}$} (A1_2);
    \path (A0_2) edge [->]node [auto] {$\scriptstyle{}$} (A1_2);
    \path (A1_1) edge [->]node [auto] {$\scriptstyle{}$} (A2_1);
    \path (A2_1) edge [->]node [auto] {$\scriptstyle{}$} (A2_3);
    \path (A1_2) edge [->]node [auto] {$\scriptstyle{}$} (A1_3);
    \path (A0_1) edge [->]node [auto] {$\scriptstyle{}$} (A1_1);
  \end{tikzpicture}
  \]
where $u$ is the restriction along $1:\Spec k\arr G$. The pullback of the above tower along a map $\stX\arr \Bi(G,H)$ yields exactly the tower encoded in the map $\stX\arr \Bi(G,H)$.
\end{rmk}

\begin{prop}\label{geometry of BGH}
 The stack $\Bi(G,H)$ is algebraic, locally of finite type over $k$ and has affine diagonal.
\end{prop}

\begin{proof}
The stack $\Bi(G,H)$ is algebraic and locally of finite type over the field $k$ because 
$\Homsh_k(G,\Bi H)$ is so by \cite[Theorem 3]{HR} and the fact that there is a finite and 
flat map $$\Homsh_k(G,\Bi H)\arr \Bi(G,H)$$ by Remark \ref{universal tower of BGH}.
 
Now let $P_i \arr P'_i \arr T$, $i\,=\,1,2$, be two towers $\xi_i \in \Bi(G,H)(T)$
for some affine scheme $T$, and set $I\,=\,\Isosh_{\Bi(G,H)}(\xi_1,\xi_2)$.  We need
to show that $I$ is affine. We denote by $(-)_T$ the base change to $T$. Base changing
along $P'_1\times_T P'_2 \arr T$ allows us to assume that $P'_i=G_T$. In particular,
there is a map $a\colon I\arr G_T=\Autsh^G_T(G_T)$ and we want to show that it is
affine. If $W$ is the fiber of $a$ along a map $T\arrdi g G_T$, it is enough to show
that $W$ is affine. Let $\overline P_2\arr G_T$ be the $H$-torsor base change of
$P_2\arr G_T$ along the multiplication $G_T\arr G_T$ by $g$ and set $J\,=\,
\Isosh^H_{G_T}(P_1,\overline P_2) \arr G_T$. It is easy to see that $W\,=\,W_{G_T}(J)$,
where $W_{G_T}(J)$ is the Weil restriction of $J$ along $G_T \arr T$. The scheme $J$
is affine and finitely presented over $G_T$. Since $G_T\arr T$ is a cover
using a presentation of $J$ it is easy to write $W$ as a closed subscheme of an
affine space over $T$.
\end{proof}

\begin{lem}\label{towers are ess finite}
Let $\stX$ be a pseudo-proper and inflexible algebraic stack of finite type over 
$k$, and let $\stZ\arrdi h \stY\arrdi g \stX$ be a $(G,H)$-tower of torsors. If $\car k 
> 0$ assume that either $g$ or $h$ is \'etale or $\dim_k \Hl^1(\stX,E)<\infty$ for every
vector bundle $E$ on $\stX$. Then $g\circ h$ is an essentially finite cover.
\end{lem}

\begin{proof} 
Consider the morphism $v\colon \stY \arr \Bi H$ corresponding to the $H$-torsor 
$h\colon \stZ \arr \stY$. One knows that $h_* \odi\stZ\,=\,v^* (k[H])$, where $k[H]$ 
denotes the regular representation. One then concludes the proof by Lemma \ref{push of ess finite is still ess 
finite} and Theorem \ref{Galois closure when G or H is etale}.
\end{proof}

\begin{lem}\label{automorphisms of a generic tower}
Let $P\arr G \arr \Spec k$ be a rational point $\xi$ of $\Bi(G,H)$ in which $G\arr \Spec k$ is the trivial $G$-torsor, and let
$p\in P(k)$ be a rational point mapping to $1\in G(k)$, and let $Q_\xi\,=\,\Autsh_{\Bi(G,H)}(\xi)$. Then
$Q_\xi$ is an affine group scheme of finite type over $k$, and there are exact sequences
 \[
     0\arr W_{G/k}(\Autsh^H_G(P)) \arr Q_\xi \arrdi \alpha G\comma \ \ 0\arr
     W' \arr W_{G/k}(\Autsh^H_G(P)) \arrdi \beta H
 \]
 where $\alpha$ forgets the automorphism of $P$, $W_{G/k}$ denotes the Weil
restriction and $\beta$ is the evaluation at $1\colon \Spec k \arr G$.
 
 There is a fully faithful map $\Bi Q_\xi \arr \Bi(G,H)$ sending the trivial torsor
to $\xi$ and whose image consists of the substack of towers locally isomorphic to $\xi$.
\end{lem}

\begin{proof}
The scheme $Q_\xi$ is affine of finite type by Proposition \ref{geometry of BGH}.
The first sequence is clear. The map $\beta$ is well-defined because the point
$p\in P(k)$ gives an $H$-equivariant isomorphism between $H$ and the base change of
$P\arr G$ along the identity: the base change of $\Autsh^H_G(P) \arr G$ along the
identity is $\Autsh^H_k(H)=H$. The last claim is standard.
\end{proof}

More can be said in the case of the trivial tower.

\begin{lem}\label{automorphisms of the trivial tower}
 Let $Q$ be the sheaf of automorphisms of the trivial tower $$G\times H\arr G\arr
\Spec k$$ in $\Bi(G,H)$. Then $Q$ is an affine group scheme of finite type,
 \[
 Q=W_G(H)\ltimes G
 \]
 where $W_G(H)$ is the Weil restriction of the group scheme $H\times_kG$ over $G$ along $G\arr 
\Spec k$ with $G$ acting on $W_G(H)$ via automorphisms of the base. Evaluation at
$1\in G$ yields a map $W_G(H)\arr H$ and, if $W'$ is the kernel, then
 \[
 W_G(H)= W' \ltimes H\, .
 \]
 The fully faithful map $\Bi Q\arr \Bi(G,H)$ of Lemma \ref{automorphisms of a generic tower} corresponds to the pointed tower of Nori-reduced torsors
 \[
 \Bi W' \arr \Bi W_G(H) \arr \Bi Q\, .
 \]
If $G$ or $H$ is \'etale over $k$ then $Q$ is a finite group scheme.
\end{lem}

\begin{proof}
Consider the maps $\alpha$ and $\beta$ defined in Lemma
\ref{automorphisms of a generic tower}. We have $P=G\times H$, so that
$$\Autsh^H_G(P)=H\times G \arr G\, .$$ Moreover it is easy to see that
both $\alpha$ and $\beta$ are surjective. The map $$G\arr Q\, ,
\ \  g\longmapsto (t_g \times \id_H,t_g)\, ,$$ where $t_g$ is the multiplication by
$g$, produces the first decomposition. The map $H(T)\arr H(G\times T)=W_G(H)(T)$
produces the second decomposition.

The claim about the tower of $\Bi Q$ is an easy consequence of Lemma \ref{key product 
gerbes}.
 
Now assume that $G$ or $H$ are \'etale. We can also assume that $k$ is 
algebraically closed, so that $G$ or $H$ are constant. If $G$ is constant then 
$W_G(H)=H^{\#G}$ is a finite scheme and therefore $Q$ is finite. Thus assume $H$ 
constant. The map $G_\red\arr G$, where $(-)_\red$ denotes the reduction, is a 
nilpotent closed immersion and, since $H$ is \'etale, it follows that the map 
$W_G(H)\arr W_{G_\red}(H)$ is an isomorphism. But $G_\red$ is equal to the \'etale quotient $G_\et$ of $G$ (because $k$ is perfect) and again we
conclude that $W_{G_\red}(H)$ is finite.
\end{proof}

\begin{prop}\label{etale case}
If $G$ or $H$ is \'etale then the map $\Bi Q\arr \Bi(H,G)$ is an equivalence, 
where $Q$ is the sheaf of automorphisms of the trivial tower $G\times H\arr G\arr 
\Spec k$ in $\Bi(G,H)$. In particular, $\Bi(G,H)$ is a finite neutral gerbe over 
$k$.
\end{prop}

\begin{proof}
In view of Lemma \ref{automorphisms of the trivial tower} it suffices to show 
that any tower is fpqc locally trivial.

Let $P\arr P'\arr U$ be a tower over an 
affine scheme. Using base changing along $P\arr U$ we
may assume this map has a section, 
so that, in particular, $P'=U\times G$. We can also assume that $k$ is 
algebraically closed, so that $G$ or $H$ is constant. If $G$ is constant, $P\arr 
U\times G$ is given by a number $\#G$ of $H$-torsors over $U$ and, trivializing those torsors, 
one gets a trivialization of $P\arr U\times G$. Now consider $H$ to be \'etale. Since 
$P\arr U\times G$ is \'etale and $U\times G_\red \arr U\times G$ is a nilpotent 
closed immersion we conclude that $P\arr U\times G$ has a section if and only if its 
restriction to $U\times G_\red$ is trivial. Thus we reduce to the known case 
where $G=G_\red=G_\et$ is \'etale.
\end{proof}

We now move to the problem of finding a Galois closure for a given tower of torsors.

\begin{defn}\label{Galois closure}
Let $\stX$ be a fibered category, and let $\stZ\arr\stY\arr\stX$ be a $(G,H)$-tower of
torsors. A \textit{Galois envelope} for the $(G,H)$-tower consists of the following data:
\begin{itemize}
\item a finite group scheme $\shG$ with homomorphisms of group schemes $\alpha\colon \shG\arr
G$ and $\Ker(\alpha)\arr H$,

\item a $\shG$-torsor $\stP\arr \stX$ together with a factorization $\stP\arr \stZ$
such that $\stP\arr \stY$ is $\shG$-equivariant and $\stP\arr \stZ$ is
$\Ker(\alpha)$-equivariant.
\end{itemize}
We say that the Galois envelope is Nori-reduced if the $\shG$-torsor $\stP\arr \stX$ is Nori-reduced. If $\sZ\arr \sX$ is an essentially finite cover, and if the Galois envelope $\sP\arr \sX$ coincides with the Galois closure of the essentially finite cover, then we call $\sP\arr\sX$ the Galois closure for the $(G,H)$-tower.
\end{defn}

Now we get a beautiful Galois envelope for the case when either $G$ or $H$ is \'etale:

\begin{thm}\label{Galois closure when G or H is etale}
 If $\stX$ is a category fibered over $k$ and $G$ or $H$ is \'etale then
 \begin{enumerate}
     \item every $(G,H)$-tower of torsors $\stZ\arrdi h \stY\arrdi g \stX$ is an essentially
           finite cover and,
     \item it admits a canonical Galois envelope $\eta\colon\stP\arr \stX$, $\lambda\colon\sP\arr\sZ$,  where $\eta$ is a torsor under the finite 
          group scheme $Q$ representing the sheaf of automorphisms of the trivial 
          tower $G\times H\arr G\arr \Spec k$; 
     \item the map $\lambda$ is a torsor under a finite subgroup scheme of $Q$;
     \item if $\sX$ is pseudo-proper 
          and inflexible, then $\eta\colon\sP\arr \sX$ is the Galois closure for the $(G,H)$-tower if and only if
          the corresponding map $\sX\arr\Bi(G,H)=\Bi Q$ is Nori-reduced.
 \end{enumerate}
\end{thm}

\begin{proof} By Lemma \ref{automorphisms of the trivial tower} and Proposition
\ref{etale case} we have $\Bi(G,H)=\Bi Q$, with universal tower $\Bi W'\arrdi b
\Bi W_G(H) \arrdi a \Bi Q$. The $Q$-torsor $\Spec k\arr \Bi Q$ with splitting
$\Spec k\arr \Bi W'$ gives a Galois closure of the universal tower. Now for any $(G,H)$-tower $\sX\arr \Bi(G,H)=\Bi Q$, the Galois envelope $\sP\arr\sX$ is just the pullback of $\Spec k\arr \Bi Q$ along $\sX\arr \Bi(G,H)=\Bi Q$.
Since $(g\circ h)_*\sO_{\sZ}$ is
the pullback of $(a\circ b)_*\sO_{\Bi W'}$ along $\sX\arr\Bi(G,H)=\Bi Q$, it is essentially finite. This finishes (1), (2), (3).

Let $\langle (a\circ b)_*\sO_{\Bi W'}\rangle$ be the full Tannakian
subcategory of $\Bi Q$ generated by $(a\circ b)_*\sO_{\Bi W'}$ with
Galois group $Q'$. Then the surjection $Q\twoheadrightarrow G$ factors
as $Q\twoheadrightarrow Q'\twoheadrightarrow G$. Let $W$ be the inverse
image of $Q'$ under $W_G(H)\hookrightarrow Q$. Then we have a Cartesian
diagram $$\xymatrix{\Bi W'\ar[d]^-b&\\ \Bi W_G(H)\ar[d]^a\ar[r]^-c& \Bi W
\ar[d]\\ \Bi Q\ar[r]&\Bi Q'}$$
By Lemma \ref{vector bundles of an extension} we see that $b_*\sO_{\Bi W'}$ is
a pullback of a vector bundle on $\Bi W$. By the same argument of Lemma 
\ref{trivial section iff gerbe} we can complete the arrows $b,c$ to a Cartesian diagram where the 
northeastern Vertex is a gerbe. This implies that $\Ker(W_G(H))\arrdi c W$ is contained in $W'$, so $\Bi W_G(H)\arr \Bi H$ factors through
$c$. In this way we obtain a tower for the stack $\Bi Q'$. By the universal property of the tower on $\Bi Q$ we conclude $Q=Q'$. Thus
the regular representation of $Q$ (i.e. the pushforward of $\sO_{\Spec k}$ along $\Spec k\arr\Bi Q$) is a 
subquotient of $f((a\circ b)_*\sO_{\Bi W'},(a\circ b)_*\sO_{\Bi W'})$ for some $f(X,Y)\in\N[X,Y]$. Conversely, $(a\circ b)_*\sO_{\Bi W'}$ is obviously a subobject of the
regular representation of $Q$.

If $\sX$ pseudo-proper and inflexible, then the above shows that the 
Tannakian subcategory of $\Ess(\sX)$ generated by $\eta_*\sO_\stP$ is the same as the
essential image of the pullback functor $\Vect(\Bi Q)\arr \Ess(\sX)$
(which is the Tannakian subcategory generated by $(g\circ h)_*\sO_\sZ$).
Thus $\sP\arr\sX$ is identified with the Galois closure $\sP_f$
of $\sZ\xrightarrow{f\coloneqq g\circ h}\sX$ constructed in Theorem \ref{closure of an essentially finite cover} if and only 
if $\sX\arr\Bi Q$ is Nori-reduced.
\end{proof}

The following lemma shows that a pointed tower $\stZ\arr\stY\arr\stX$ has a Galois closure.

\begin{thm}\label{Galois closure for essentially finite towers}
Let $\stX$ be a pseudo-proper and inflexible algebraic stack of finite type over $k$ and
$\stZ\xrightarrow{h} \stY\xrightarrow{g} \stX$ a pointed tower of torsors. If $\car k > 0$, assume that either $g$ is
\'etale or $\dim_k \Hl^1(\stX,E)<\infty$ for every vector bundle $E$ on $\stX$.
Then the following hold:
\begin{enumerate}
\item $f\colon \stZ\arr \stX$ is essentially finite,

\item the map $\omega\colon \stX\arr\Bi(G,H)$ corresponding to the tower factors through the monodromy gerbe
$\stX\arrdi\phi \Gamma\,=\,\Bi G_f$ of $f_*\odi\stZ$ in $\EF(\Vect(\stX))$ and

\item the $G_f$-torsor $\stP_f\arr \stX$ and the factorization $\stP_f\arr \stZ$ introduced
in Theorem \ref{closure of an essentially finite cover} define a  Galois closure for the $(G,H)$-tower
$\stZ\arr\stY\arr\stX$.
\end{enumerate}
Moreover, the group scheme $G_f$ is a finite subgroup of the affine and finite type $k$-group scheme
$\Autsh_{\Bi(G,H)}(\omega(x))$, where $x$ is the given rational point of $\stX$.
\end{thm}

\begin{proof}
The cover $f\colon \stZ\arr \stX$ is essentially finite by Lemma \ref{towers are ess finite}.
We want to extend the given tower along $\stX\arrdi\phi \Bi G_f$ as in the diagram below.
Using the notation of Theorem \ref{closure of an essentially finite cover} (but the map $f\colon \sY\arr\sX$ there is replaced by $f\colon \sZ\arr\sX$ in our case), according to
Lemma \ref{morphisms of torsors} there exists a morphism of group schemes $G_f \longrightarrow G$
inducing the morphism $\stP _f \arr \stY$. By Lemma \ref{trivial section iff gerbe} the cover
$\stZ\arr \stX$ extends to a cover $\Delta \arr \Bi G_f$. Let $U\coloneqq \Bi G_f\times_{\Bi G}\Spec k$. We have the following Cartesian diagrams:
\[
  \begin{tikzpicture}[xscale=1.5,yscale=-1.2]
    \node (A0_0) at (0, 0) {$\stP_f$};
    \node (A0_1) at (1, 0) {$\stZ$};
    \node (A0_2) at (2, 0) {$\stY$};
    \node (A0_3) at (3, 0) {$\stX$};
    \node (A1_0) at (0, 1) {$\Spec k$};
    \node (A1_1) at (1, 1) {$\Delta$};
    \node (A1_2) at (2, 1) {$U$};
    \node (A1_3) at (3, 1) {$\Bi G_f$};
    \path (A0_1) edge [->]node [auto] {$\scriptstyle{}$} (A1_1);
    \path (A0_0) edge [->]node [auto] {$\scriptstyle{}$} (A0_1);
    \path (A0_1) edge [->]node [auto] {$\scriptstyle{}$} (A0_2);
    \path (A1_0) edge [->,dashed]node [auto] {$\scriptstyle{a}$} (A1_1);
    \path (A0_3) edge [->]node [auto] {$\scriptstyle{}$} (A1_3);
    \path (A1_1) edge [->,dashed]node [auto] {$\scriptstyle{b}$} (A1_2);
    \path (A0_2) edge [->]node [auto] {$\scriptstyle{}$} (A1_2);
    \path (A0_0) edge [->]node [auto] {$\scriptstyle{}$} (A1_0);
    \path (A0_2) edge [->]node [auto] {$\scriptstyle{}$} (A0_3);
    \path (A1_2) edge [->]node [auto] {$\scriptstyle{}$} (A1_3);
  \end{tikzpicture}
  \]
where the dashed arrows come from the existing maps $\sP_f\arr\sZ$ (for $a$) and $\sZ\arr\sY$ (for $b$) and the full faithfulness of $\Vect(\Bi G_f)\arr \Vect(\stX)$. We must equip $b\colon \Delta\arr U$ with a compatible structure of $H$-torsor. Notice that, by
Lemma \ref{vector bundles of an extension}, a vector bundle on $\stY$ whose pullback to $\stZ$ is
free (and thus comes from a vector bundle on $\Delta$) comes from a vector bundle on $U$.
Moreover Lemma \ref{vector bundles of an extension} also tells us that $\Vect(U)\arr \Vect(\stY)$ is fully faithful. This shows that we get a factorization
\[
\Vect(\Bi H)\arr \Vect(U)\arr \Vect(\stY)
\]
and, by Tannakian duality, a factorization $\stY\arr U\arr \Bi H$. This determines an $H$-torsor $\Delta'\arr U$ extending $\stZ\arr \stY$. Since $\Vect(U)\arr \Vect(\stY)$ is fully faithful one concludes that $\Delta'\simeq \Delta$ over $U$.

Let $K$  be the kernel of $G_f\arr G$. The map $\Spec k \arr U$ factors through a closed immersion $\Bi K\arr U$ and the composition $\Bi K\arr U\arr \Bi H$ preserves trivial torsors,
meaning it is induced by a homomorphism $K\arr H$. It is easy to show that all the data constructed define a Galois envelope of the original tower with Galois group $\sG=G_f$.

For the last claim, set $\xi=\omega(x)$ and $Q=\Autsh_{\Bi(G,H)}(\xi)$. By hypothesis the tower
$\xi$ is pointed and therefore, by Lemma \ref{automorphisms of a generic tower}, the group scheme $Q$ is affine and of finite type and there is a fully faithful map $\Bi Q\arr \Bi(G,H)$ whose essential image is the full substack of $\Bi(G,H)$ of towers which are fpqc locally isomorphic to $\xi$. Since $\omega$ factors as $\stX\arr \Bi G_f\arr \Bi(G,H)$, all objects in the image of $\Bi G_f\arr \Bi(G,H)$ are locally isomorphic to $\xi$. Thus the previous morphism factors through $\Bi G_f \arr \Bi Q$. This map preserves trivial torsors and it is therefore induced by a map $G_f\arrdi q Q$. Let $\shG$ be the image of $q$. The factorization
\[
\stX\arr \Bi G_f \arr \Bi \shG \arr \Bi Q \subseteq \Bi(G,H)
\]
tells us that the tower over $\stX$ extends to a tower over $\Bi \shG$ and therefore $f_*\odi\stZ$
comes from a vector bundle on $\Bi \shG$. But $\Bi G_f$ is exactly the monodromy gerbe of
$f_*\odi\stZ$ in $\EF(\Vect(\stX))$, which implies that $G_f = \shG \subseteq Q$.
\end{proof}

\begin{proof}[Proof of Theorem \ref{main theorem for Galois closure}]
The closure we consider is the one in Theorem \ref{Galois closure for essentially finite towers}.
Statement (1) follows by applying Theorem \ref{main thm Nori gerbe of ess finite cover} and
lemma \ref{morphisms of torsors} to both $\stX$ and $\stY$ as bases. Statement (2) follows from the
last statement  of Theorem \ref{Galois closure for essentially finite towers}.
\end{proof}

\section{The $\SS$-fundamental gerbe of essentially finite covers}

The aim of this section is to prove Theorem \ref{main thm for the S-gerbe}.
We start by introducing the $\SS$-fundamental gerbe, which generalizes the notion of $\SS$-fundamental group (see \cite{BPS}, \cite{L1}, \cite{L2}). 

\begin{defn}\label{Nori semistable}
A vector bundle $V$ on a fibered category $\stX$ is called \emph{Nori semistable} if for all 
smooth projective curves $C$ over an algebraically closed field and all maps $i\colon C\arr 
\stX$ the pullback $i^*V$ is semistable of degree $0$.
 
 We denote by $\Ns(\stX)$ the full subcategory of $\Vect(\stX)$ of Nori semistable vector bundles.
 
 The $\SS$-fundamental gerbe of a fibered category $\stX$ over $k$ is an affine gerbe 
$\Pi$ over $k$ together with a map $u\colon \stX\arr \Pi$ whose pullback $u^*\colon 
\Vect(\Pi)\arr \Vect(\stX)$ is fully faithful with essential image $\Ns(\stX)$. The 
$\SS$-fundamental gerbe is unique and is denoted by $\Pi^\SS_{\stX/k}$ when it exists.
 
If $\stX$ has an $\SS$-fundamental gerbe $\Psi\colon \stX\arr \Pi^\SS_{\stX/k}$ and $x\in 
\stX(k)$, then the $\SS$-fundamental group $\pi^\SS(\stX/k,x)$ of $(\stX,x)$ over $k$ is the 
sheaf of automorphisms of $\psi(x)\in \Pi^\SS_{\stX/k}$.
 
We will usually drop the $/k$ when $k$ is clear from the context.
\end{defn}

\begin{rmk}\label{when X has an S-fundamental gerbe}
A fibered category $\stX$ over $k$ has an $\SS$-fundamental gerbe if and only if 
$\Hl^0(\odi\stX)=k$ and $\Ns(\stX)$ is an abelian subcategory of $\QCoh(\stX)$. The ``only if'' is 
clear. For the converse observe that $\Ns(\stX)$ is a rigid monoidal category. If it is also an 
abelian subcategory of $\QCoh(\stX)$ then $\Ns(\stX)$ is $k$-Tannakian and the map 
$\Ns(\stX)\arr \Vect(\stX)$ sends exact sequences to exact sequences in $\QCoh(\stX)$. By 
Tannakian duality $\Ns(\stX) \simeq \Vect(\Pi)$, where $\Pi$ is an affine gerbe, and by Remark 
\ref{Tannakian reconstruction}, the inclusion $\Ns(\stX)\subseteq \Vect(\stX)$ is realized as 
the pullback of a map $\stX\arr \Pi$, that is $\Pi=\Pi^\SS_\stX$.
\end{rmk}

\begin{rmk}\label{NSX Tannakian implies inflexible}
If $\stX$ is a fibered category with an $\SS$-fundamental gerbe, then its profinite quotient is
a Nori fundamental gerbe. In particular $\stX$ is inflexible.
Indeed since finite vector bundles on $\stX$ are Nori semistable, we have 
 \[
 \EF(\Vect(\stX)) = \EF(\Ns(\stX))
 \]
is also a $k$-Tannakian category by Theorem \ref{monodromy essentially finite} and that it
is an abelian subcategory of $\QCoh(\stX)$. From Remark \ref{Tannakian reconstruction} and
Remark \ref{pullback from finite stack} it follows that the affine gerbe associated to
$\EF(\Ns(\stX))$, which is the profinite quotient of $\Pi^\SS_\stX$, is a Nori fundamental
gerbe for $\stX$.
\end{rmk}

\begin{rmk}\label{functoriality and descent for Nori semistable}
Let $\phi\colon \stX'\arr \stX$ be a map of fibered categories and $F\in \Vect(\stX)$. If $\shF$ 
is Nori semistable then $\phi^*\shF$ is Nori semistable too. The converse holds if $\phi$ is 
representable (by a scheme), proper and surjective. Indeed one can assume that $\stX$ is a proper, smooth, 
integral curve over an algebraically closed field $k$ and must prove that $\shF$ is semistable 
of degree $0$. Considering a closed point in the generic fiber of $\phi$ and taking the 
normalization of its closure one can moreover assume that $\stX'$ is also a proper, smooth, 
integral curve over $k$. In particular $\phi$ is a cover. In this case the result follows 
because the pullback of a subbundle destabilizing $\shF$ actually destabilizes $\phi^*\shF$.
\end{rmk}

\begin{ex}
 If $X$ is a smooth, geometrically connected and geometrically projective scheme over $k$
then $X$ has an $\SS$-fundamental gerbe over $k$ (see \cite{BPS}, \cite{L1}, \cite{L2}). Recently, it is shown
in \cite[Theorem 6.7]{BHD} that if $X$ is a reduced algebraic $k$-scheme which is connected by proper chains (see \cite[Definition 6.1]{BHD}), where $k$ is algebraically closed, then
$X$ has an $\SS$-fundemental gerbe over $k$.
\end{ex}

\begin{prop}
 An affine gerbe $\Gamma$ over $k$ has an $\SS$-fundamental gerbe over $k$.
\end{prop}

\begin{proof}
In view of Remark \ref{when X has an S-fundamental gerbe} we need to show that $\Ns(\Gamma)$ is an
abelian subcategory of $\QCoh(\Gamma)$. So let
 \[
 0\arr K \arr F_1 \arr F_2 \arr Q \arr 0
 \]
 be an exact sequence in $\QCoh(\Gamma)$ with $F_1,F_2\in \Ns(\Gamma)$. We must show that
$K,Q\in \Ns(\Gamma)$. Let $i\colon C\arr \Gamma$ be a map from a smooth projective curve over
some algebraically closed field. Since $\Gamma$ is a gerbe, both $K$ and $Q$ are vector bundles. Thus
$i^*K$ and $i^*Q$ are respectively kernel and cokernel, in $\QCoh(C)$, of a homomorphism between Nori semistable
vector bundles on $C$. Since $\Ns(C)$ is an abelian subcategory of $\QCoh(C)$, it follows that $i^*K$ and $i^*Q$ are in $\Ns(C)$. 
\end{proof}

\begin{lem}\label{push of Nori semistable are Nori semistable}
Let $\stX$ be a pseudo-proper and inflexible category over $k$, and let $f\colon \stY\arr \stX$ be an essentially finite cover. Then
 $$
 \Ns(\stY)=\{ V \in \Vect(\stY) \st f_*V\in \Ns(\stX) \}
 $$
\end{lem}

\begin{proof}
Given $F\in \Vect(\stY)$ we have to prove that $$F\in \Ns(\stY) \iff f_*F\in \Ns(\stX)\, .$$
Nori semistability is tested on curves. Thus we can assume that $\stX=C$ is a smooth, integral,
projective curve over an algebraically closed field $k$: for ``$\Longrightarrow$'' we know that
$F\in \Ns(\stY)$ and we must prove that $f_*F\in \Ns(C)$; for ``$\Longleftarrow$'' we have a
section $C\arrdi i \stY$, we know that $f_*F\in \Ns(C)$ and we must prove that $i^*F\in \Ns(C)$.
Here we are using the following: since $\stX$ is pseudo-proper and inflexible, the pullback
of $f_*\odi\stX$ along the curve $C\arr \stX$ is still essentially finite;
see Remark \ref{pullback from finite stack}.

Let $C\arr \Gamma$ be the monodromy gerbe of $f_*\odi\stY$ in $\EF(\Vect(C))$ and $\Delta\arr 
\Gamma$ the extension given in Lemma \ref{trivial section iff gerbe}. We have $\Gamma=\Bi G$ for some 
finite group scheme $G$ so that the map $C\arr \Bi G$ is given by a $G$-torsor $\stQ\arr C$. Let 
$D$ be the normalization of an irreducible component of $\stQ$ surjecting onto $C$. It follows 
that $g\colon D\arr C$ is a surjective cover. Since a vector bundle on $C$ is Nori semistable if 
and only if its pullback via $g$ is so (see Remark
\ref{functoriality and descent for Nori semistable}), 
we can replace $C$ by $D$, that is assume that $C\arr \Bi G$ factors through $\Spec k$. Since 
the cover $\stY\arr C$ extends to $\Bi G$ we know that $f\colon \stY=C\times A \arr C$ is the 
projection, where $A/k$ is a finite $k$-algebra. Splitting $\stY$ according to a decomposition 
of $A$ we can moreover assume $A$ local. In the case ``$\Longleftarrow$'' the inclusion $i\colon 
C\arr \stY=C\times A$ is induced by $A\arr k$. This map is also defined in the case 
``$\Longrightarrow$'' and we denote it with the same symbol $i$.

We may replace $C$ by another test curve. Hence it is enough to prove that $f_*F$ is
semistable of degree $0$ if and only if $i^*F$ is so. From Lemma
\ref{decomposition series lemma} we obtain a sequence of surjective homomorphisms of vector bundles
$$
f_*F=\shG_N \arr \shG_{N-1} \arr \cdots \arr \shG_1 \arr \shG_0=0
$$
such that $\Ker(\shG_l \arr \shG_{l-1})\simeq i^*F$. 
By induction we have
\[
\det(f_*F)\simeq (\det i^*F)^N
\]
so that $f_*F$ has degree $0$ if and only if $i^*F$ has degree $0$. Again by induction we also see
that all $\shG_l$ have the same slope as that of $i^*F$. In particular
if $f_*F$ is semistable so is $i^*F\subseteq f_*F$. The
converse is deduced from the following fact: if $0\arr E'\arr E \arr E'' \arr 0$ is an exact sequence of vector bundles on $C$ with equal slope then $E$ is semistable if $E'$ and $E''$ are semistable. 
\end{proof}

\begin{proof}[Proof of Theorem \ref{main thm for the S-gerbe}]
 The first claim follows from Remark \ref{NSX Tannakian implies inflexible}. In particular by Theorem \ref{main thm Nori gerbe of ess finite cover} we have $2$-Cartesian diagrams
   \[
  \begin{tikzpicture}[xscale=1.8,yscale=-1.1]
    \node (A0_0) at (0, 0) {$\stY$};
    \node (A0_1) at (1, 0) {$\Pi$};
    \node (A0_2) at (2, 0) {$\Pi^\NN_\stY$};
    \node (A1_0) at (0, 1) {$\stX$};
    \node (A1_1) at (1, 1) {$\Pi^\SS_\stX$};
    \node (A1_2) at (2, 1) {$\Pi^\NN_\stX$};
    \path (A0_0) edge [->]node [auto] {$\scriptstyle{v}$} (A0_1);
    \path (A0_1) edge [->]node [auto] {$\scriptstyle{}$} (A0_2);
    \path (A1_0) edge [->]node [auto] {$\scriptstyle{u}$} (A1_1);
    \path (A0_2) edge [->]node [auto] {$\scriptstyle{}$} (A1_2);
    \path (A1_1) edge [->>]node [auto] {$\scriptstyle{}$} (A1_2);
    \path (A0_0) edge [->]node [auto] {$\scriptstyle{}$} (A1_0);
    \path (A0_1) edge [->]node [auto] {$\scriptstyle{}$} (A1_1);
  \end{tikzpicture}
  \]
Since $\Pi^\SS_\stX\arr \Pi^\NN_\stX$ is a quotient it follows that $\Pi$ is a gerbe. As $\Ns(\stX)$
is a full subcategory of $\Vect(\stX)$, by Lemma \ref{fully faithful on Vect means pushforwars O is O},
Lemma \ref{vector bundles of an extension} and Lemma \ref{push of Nori semistable are Nori semistable}
we have that $\Vect(\Pi)\arr \Vect(\stY)$ is fully faithful with essential image
$\Ns(\stY)$. Thus we have $\Pi=\Pi^\SS_\stY$. The last claim follows from the Cartesian diagram
and Lemma \ref{key product gerbes}.
\end{proof}

\section{Counterexamples}

In this section we collect various examples. We start by showing that, under the assumption of 
Theorem \ref{main thm Nori gerbe of ess finite cover}, the condition $\Hl^0(\odi\stY)=k$ in 
general does not imply that $\stY$ is inflexible, even when $f_*\odi \stY$ has \'etale monodromy 
gerbe (e.g. if $\car k = 0$).

\begin{ex}\label{counterexample connected cover non inflexible}
 Let $k$ be an algebraically closed field. We show an example of an elliptic curve $X$ over $k$ with an essentially finite cover $f\colon Y \arr X$ of degree $2$ such that $\Hl^0(\odi\stY)=k$ but $Y$ is not inflexible. Clearly here $Y$ is not reduced. The monodromy gerbe of $f_*\odi Y$ is $\Bi \mu_2$.
 
 Let $X$ be an elliptic curve together with a non trivial line bundle $\shL$ such that
$\shL^{\otimes 2}\simeq \odi X$. This is the data of a Nori-reduced map $u\colon X\arr \Bi \mu_2$. Let
$A=k[\epsilon]=k[x]/(x^2)$ equipped with the $\mu_2=\Spec ( k[y]/(y^2-1))$ action
 \[
 A\arr A\otimes (k[y]/(y^2-1))\comma \epsilon \longmapsto \epsilon \otimes y
 \]
 and set $\Delta=[\Spec A/\mu_2]$. Define $Y\arr X$ with the $2$-Cartesian diagram
   \[
  \begin{tikzpicture}[xscale=1.8,yscale=-1.2]
    \node (A0_0) at (0, 0) {$Y$};
    \node (A0_1) at (1, 0) {$\Delta$};
    \node (A1_0) at (0, 1) {$X$};
    \node (A1_1) at (1, 1) {$\Bi \mu_2$};
    \path (A0_0) edge [->]node [auto] {$\scriptstyle{v}$} (A0_1);
    \path (A1_0) edge [->]node [auto] {$\scriptstyle{u}$} (A1_1);
    \path (A0_1) edge [->]node [auto] {$\scriptstyle{}$} (A1_1);
    \path (A0_0) edge [->]node [auto] {$\scriptstyle{f}$} (A1_0);
  \end{tikzpicture}
  \]
Since $u$ is Nori-reduced we have $u_*\odi X\simeq \odi{\Bi\mu_2}$ and $v_*\odi Y\simeq\odi\Delta$ by flat base change. In particular 
$$
\Hl^0(\odi Y)=\Hl^0(\odi\Delta)=A^{\mu_2}=k
$$
where the last equality follows from a direct computation. On the other hand, $Y$ is not inflexible
because $v_*\odi Y\simeq\odi\Delta$; this implies that $v\colon  Y\arr \Delta$ does not factor through a gerbe.
\end{ex}

We now show examples of towers without a Galois envelope. The lemma below will be our
method to exclude that a given tower has a Galois envelope.

A torsor under a finite group scheme $G$ over $k$ is called \emph{minimal} if it does not come from a torsor under a proper subgroup of $G$. For instance Nori-reduced torsors are minimal.

\begin{lem}\label{Galois closure and gerbes}
Let $\stX$ be a fibered category, and let $\stZ\arr \stY\arr \stX$ be a
$(G,H)$-tower of torsors with $\stY\arr \stX$ minimal. If the tower has a Galois envelope
with group $\shG$ then the map $\stX\arr \Bi(G,H)$ factors through a map $\Bi \shG \arr \Bi(G,H)$.
\end{lem}

\begin{proof}
Let $\stP\arr \stX$ and $\stP\arr \stZ$ be the Galois envelope. We must show that there is a tower
over $\Bi\shG$ whose pullback along $\stX\arr \Bi\shG$ is the original tower. Since $\stY\arr \stX$
is minimal the map $\shG\arr G$ is surjective. By Lemma \ref{key product gerbes} the $G$-torsor over
$\Bi\shG$ induced by $\shG\arr G$ is $\Bi \shH \arr \Bi\shG$, where $\shH$ is the kernel of
$\shG\arr G$. The pullback of the $G$-torsor $\Bi\shH\arr\Bi\shG$ along $\stX\arr \Bi\shG$ is, by
construction, $\stY\arr \stX$. The map $\stY\arr \Bi \shH$ is given by the $\shH$-torsor
$\stP\arr \stY$. The homomorphism $\shH\arr H$ gives a map $\Bi\shH \arr \Bi H$ and therefore an $H$-torsor
$\stB\arr \Bi \shH$. The fact that $\stP\arr \stZ$ is $\shH$-equivariant means that
$\stZ\arr \stY$ is the $H$-torsor induced by the $\shH$-torsor $\stP\arr \stY$ along $\shH\arr H$.
This exactly means that the $H$-torsor $\stB\arr \Bi\shH$ pulls back to $\stZ\arr \stY$ along
$\stY\arr \Bi\shH$.
 \end{proof}
 
This example shows that the condition on the cohomology groups $\Hl^1$ in Theorem \ref{main theorem for Galois closure} is necessary.
 \begin{ex}\label{counterexample 1}
 Assume that $G$ and $H$ are not \'etale.
  We give an example of a pseudo-proper, inflexible and smooth algebraic stacks $\stX$ with a pointed $(G,H)$-tower of Nori-reduced torsors $\stZ\arr \stY \arr \stX$ without a Galois envelope. In particular $\stZ\arr \stX$ cannot be an essentially finite cover by
Theorem \ref{Galois closure for essentially finite towers}.
  
Using notations of Lemma \ref{automorphisms of the trivial tower} set $\stX=\Bi Q$ with the tower
$\Bi W'\arr \Bi W_G(H) \arr \Bi Q$. If this tower has a Galois envelope then by 
Lemma \ref{Galois closure and gerbes} the map $\Bi Q\arr \Bi(G,H)$ factors through a finite gerbe. Since the map $\Bi Q\arr \Bi(G,H)$ is fully faithful, this means that $Q$ has to be a finite group scheme. Thus we must show that $Q$ is not a finite group scheme. In particular we can
assume $k$ to be algebraically closed, so that $G$ is a disjoint union of copies of the connected component $G_0$.
In particular $W_G(H)=(W_{G_0}(H))^{\#G_\et}$, so that we can assume $G$ local but not trivial.
Moreover there is an injective map $W_G(H_0)\arr W_G(H)$, where $H_0$ is the connected component of $H$. 
Thus we can also assume that $H$ is local but not trivial. If $k[\epsilon]=k[x]/(x^2)$ there is a map
$k[\epsilon]\subseteq k[G]$. Thus one gets an injective map of group schemes $W_U(H)\arr W_G(H)$ where $U=\Spec k[\epsilon]$. Similarly one can find a closed embedding $U\arr H$, which yields a monomorphism of schemes $W_U(U)\arr W_U(H)$. Moreover there is a monomorphism
  \[
  \A^1(B)\arr W_U(U)(B)=\Hom_{B\text{-algebras}}(B[\epsilon],B[\epsilon])\comma b \longmapsto (\epsilon \mapsto b\epsilon)
  \]
  In conclusion we find a monomorphism $\phi\colon \A^1\arr W_G(H)$. If $W_G(H)$ is finite, the image
of $\phi$ must be connected, reduced, finite and with a rational point, that is $\Spec k$, so that
$\phi$ is not a monomorphism. Therefore, $W_G(H)$ is not finite.
 \end{ex}

The next example shows the importance of the pseudo-properness assumption on $\stX$ in 
Theorem \ref{main theorem for Galois closure}.

\begin{ex}\label{counterexample 2}
  We give an example of a smooth, integral and affine scheme $X$ over $k$ with a pointed $(G,H)$-tower of Nori-reduced torsors without a Galois envelope.
  
  Assume that $k$ is an algebraically closed field of characteristic $2$ and let $H=\mu_2$ and $G$ be either $\mu_2$ or $\alpha_2$. Recall that if $B$ is a $k$-algebra and $b\in B^*$ then $B[x]/(x^2-b)$ has an action by $\mu_2$ and an action by $\alpha_2$ and it is a torsor over $B$ for both actions. Since $k$ is algebraically closed and $\alpha_2$ and $\mu_2$ are simple we have that $B[x]/(x^2-b)$ is Nori-reduced if $b$ is not a square in $B$ and it is trivial otherwise.
  
  Let $K$ be the separable closure of the field of fractions $k(t)$ and consider 
  \[
  K_1=K[x]/(x^2-t) \text{ and } K_2=K[x,y]/(x^2-t,y^2-x)
  \]
  The rings $K_1\subseteq K_2$ are fields. Thus $\Spec K_2\arr \Spec K_1 \arr \Spec K$ is a $(G,H)$-tower of non trivial torsors which defines a map $\xi\colon \Spec K\arr \Bi(G,H)$.
  Consider a smooth map $X\arr \Bi(G,H)$ from a connected affine scheme and whose image contains the point $\xi$. We claim that the corresponding tower is pointed and Nori-reduced. It is pointed because $k$ is algebraically closed. It is Nori-reduced because, since $K$ is separably closed, the map $\xi$ factors through a map $\Spec K\arr X$ and the torsors in the tower $\xi$ are not trivial.
  
  Let $\alpha\colon X\arr \Bi(G,H)$ and $\beta\colon Y\arr \Bi(G,H)$ be smooth maps from connected affine schemes and assume that their images contain $\xi$ and the trivial tower respectively. If the tower $\alpha$ does not have a Galois envelope we have our counter-example. Otherwise, by \ref{Galois closure and gerbes}, the map is (topologically) constant and $\xi$ is an open point in $\Bi(G,H)$. If $\Bi(G,H)$ is irreducible then $\xi$ would be its generic point and, since the image of $\beta$ is open, it would contain $\xi$. In this case if the tower $\beta$ has a Galois envelope it would follow that $\xi$ is the trivial tower, that is $x\in K_1$ become a square extending the field $K$, which is not true.

  Thus it is enough to show that $\Bi(G,H)$ is a smooth and connected algebraic stack.   
  Consider the tower $\stW\arr \Homsh_k(G,\Bi H)\arr \Bi(G,H)$ described in \ref{universal tower of BGH}.
  It is enough to show that $\stW$ is a smooth connected algebraic stack. The objects of $\stW(B)$ are $H$-torsors over
  $G\times B$ with a trivialization over $\Spec B\arrdi 1 G\times B$. We think about $\mu_2$-torsors as line bundles
  with an isomorphism between its square and the trivial bundle. Set also $B[\epsilon]=B[x]/(x^2)=B[G]$. Thus $\stW(B)$ 
  is the stack of triples $(L,\phi,\psi)$ where $L$ is a $B[\epsilon]$-line bundle, $\phi\colon L^2\arr B[\epsilon]$ is 
  an isomorphism and $\psi\colon L/\epsilon L \arr B$ is an isomorphism such that $\psi^{\otimes 2} \equiv \phi$ modulo 
  $\epsilon$. If $L=B[\epsilon]$ then $\phi=a+b\epsilon \in B[\epsilon]^*$ and $\psi=c\in B^*$ with $a=c^2$ and,
  up to an isomorphism in $\stW(B)$, one can always assume $c=1$. Thus the map $\A^1\arr \stW$, mapping $b\in \A^1(B)$ 
  to $(B[\epsilon],1+b\epsilon,1)$, is an epimorphism in the Zariski topology. In particular $\A^1\arr \stW$ is an fppf 
  covering if  $\A^1 \times_\stW \A^1 \arr \A^1$ is. A direct computation shows that an isomorphism 
  $(B[\epsilon],1+b\epsilon,1)\arr (B[\epsilon],1+c\epsilon,1)$ exists if and only if $b=c$ and in this case it is the
  multiplication by $1+\lambda\epsilon$ for $\lambda\in B$. This means that $\A^1 \times_\stW \A^1 \arr \A^1$ coincides
  with the projection $\pr_1\colon\A^1\times_k \A^1\arr \A^1$ which is an fppf covering.
 \end{ex}

We conclude the section with an example showing that Corollary \ref{main thm Nori 
gerbe pointed case} (and thus Theorem \ref{main thm Nori gerbe of ess finite 
cover}) as well as Lemma \ref{push of ess finite is still ess finite} fails without 
the finiteness assumption on the first cohomology group of vector bundles.

\begin{ex}\label{counterexample no H1 exact sequence}
We give an example of a Nori-reduced torsor $f\colon \stY\arr \stX$ between 
pseudo-proper, inflexible and smooth algebraic stacks over $k$ with $y\in \stY(k)$ 
such that the following hold:
\begin{itemize}
\item $f_*$ does not map essentially finite vector bundles to essentially 
finite vector bundles,

\item $\pi^\NN(\stX,f(x))$ is finite,

\item $f\colon \stY\arr \stX$ is 
the universal torsor, but 

\item $\pi^\NN(\stY,y)$ is not trivial.
\end{itemize} 

Let $k$ be a field of characteristic $2$, $H=\mu_2$, and let $G$ be either $\mu_2$ or 
$\alpha_2$. Consider the $(G,H)$-tower of pointed Nori-reduced torsors 
$$\stZ=\Bi W'\arrdi h \stY=\Bi W_G(\mu_2)\arrdi f \stX=\Bi Q$$ introduced in Example
\ref{counterexample 1}. Since $\stZ\arr \stX$ is not essentially finite it follows 
that $h_*\odi\stZ$ is essentially finite while $f_*(h_*\odi\stZ)$ is not.
 
Consider $y$ and $x=f(y)$ as the trivial torsors. Since the Nori fundamental group of
an affine gerbe $\Bi S$ pointed at the trivial torsor is the profinite quotient $\widehat S$ of $S$, we must show that $\widehat Q=G$ and $\widehat{W_G(\mu_2)}=\mu_2$. Given a $k$-algebra $B$ set $B[\epsilon]=B[x]/(x^2)$. We have
 \[
 W_G(\mu_2)(B)=\mu_2(B[\epsilon])=\{a+b\epsilon \st a^2=1\}
 \]
 from which it is easy to conclude that $W_G(\mu_2)=\G_a \times \mu_2$. Since any
homomorphism from $\G_a$ to a profinite group is trivial we conclude that
$\widehat{W_G(\mu_2)}=\mu_2$. Similarly, denoting by $K$ the kernel of
$Q\arr \widehat Q$, we have $\G_a\subseteq K \subseteq W_G(\mu_2)$. Since $\mu_2$ is
simple we just have to check that $\G_a$ is not normal in $Q$. Note that
$G=\Spec k[\epsilon]$; for $g\in G(B)$, let $\psi_g\colon B[\epsilon]\arr
B[\epsilon]$ be the multiplication by $g$. Then,
 \[
 g\star x \star g^{-1} = \psi_g(x) \text{ for } x\in \mu_2(B[\epsilon])=W_G(\mu_2)(B)
 \]
 where $\star$ denotes the multiplication in $Q$. If $G=\mu_2$, so that $g\in B^*$ with $g^2=1$, an easy computation shows that $\psi_g(\epsilon)=g\epsilon + (g-1)$. Thus
 $$
 \psi_g(1+\epsilon)=g + g\epsilon 
 $$
 is in $\G_a$ if and only if $g=1$, which is not always the case.
 
 If $G=\alpha_2$, so that $g\in B$ with $g^2=0$, then $\psi_g(\epsilon)=\epsilon + g$. Thus
 \[
 \psi_g(1+\epsilon)=1+g + \epsilon 
 \]
 is in $\G_a$ if and only if $g=0$, which is not always the case.
\end{ex}

\end{document}